\documentclass[10.5pt]{amsart}
\usepackage{amssymb}
\usepackage{eucal}
\usepackage{graphicx}
\usepackage{color}
\usepackage[colorlinks]{hyperref}
\usepackage[active,new,noold,marker]{xrcs}
\usepackage[active]{xrcs}
\usepackage{hyperref}
\newtheorem{Definition}{Definition}[subsection]
\newtheorem{Theorem}[Definition]{Theorem}
\newtheorem{Lemma}[Definition]{Lemma}

\newtheorem{Proposition}[Definition]{Proposition}
\newtheorem{Corollary}[Definition]{Corollary}

\newtheorem{Remark}[Definition]{Remark}

\newcommand{\bt}{\begin{Theorem}}
\newcommand{\et}{\end{Theorem}}
\newcommand{\ba}{\begin{eqnarray}}
\newcommand{\ea}{\end{eqnarray}}
\newcommand{\bd}{\begin{Definition}}
\newcommand{\ed}{\end{Definition}}
\newcommand{\bp}{\begin{Proposition}}
\newcommand{\ep}{\end{Proposition}}
\newcommand{\bl}{\begin{Lemma}}
\newcommand{\el}{\end{Lemma}}
\newcommand{\br}{\begin{Remark}}
\newcommand{\er}{\end{Remark}}
\newcommand{\bpf}{\begin{proof}}
\newcommand{\epf}{\end{proof}}
\newcommand{\be}{\begin{equation}}
\newcommand{\ee}{\end{equation}}
\newcommand{\mc}{\mathcal}

\newcommand{\f}{\frac}

\newcommand{\what}{\widehat}

\newcommand{\mf}{\mathfrak}
\newcommand{\R}{\mathbb R}%
\newcommand{\C}{\mathbb C}%
\newcommand{\Z}{\mathbb Z}%
\newcommand{\lt}{\left}%
\newcommand{\rt}{\right}%
\newcommand{\e}{\varepsilon}%
\numberwithin{equation}{section}
\numberwithin{Definition}{section}

\allowdisplaybreaks[1]
\begin{document}

\title[Spectral projection]
 {Image of Schwartz Space Under Spectral Projection}

\author[Jana]{Joydip Jana}

\address{Department of Mathematics, Syamaprasad College; 92, S. P. Mukherjee Road; Kolkata- 700 026 }

\email{joydipjana@gmail.com}

\thanks{The author is thankful to Prof. S. C. Bagchi and Rudra P. Sarkar of Indian Statistical Institute for their useful suggestions}

\date{\today }
\dedicatory{}


\begin{abstract}
Let $X= G/K$ symmetric space of non compact type, where $G$ is a rank-one
connected semisimple Lie group with finite center. We shall look at the
transform $ P_\lambda f(x) = f \ast \varphi_\lambda(x)$, where, $\lambda \in \mathbb C$ and  $\varphi_\lambda$ is the elementary spherical function. We shall try to characterizes the image of  the Schwartz spaces $S^p(X) $ where $0 < p \leq 2$ under the above transform.
\end{abstract}

\maketitle

\section{\textbf{Introduction}}
 \label{sec: Introduction}
Let $X$ be the Riemannian symmetric space realized as $G / K$ where $G$ is a connected, noncompact, real rank-one semi-simple Lie group with finite center.
Let $\varphi_\lambda$
$(\lambda \in \C)$ be the elementary spherical functions of $G$. We denote $ \mc S^p(X)$ ($0 < p \leq 2$) for the $L^p$-Schwartz class functions on $X$.
 For $f \in \mc S^p(X)$ ($0<p\leq 2$) (for the case $1< p \leq 2$, $f\in L^p(X)$) we consider the transform $f
\mapsto P_\lambda f(x)=f \ast \varphi_\lambda(x)$
for each $\lambda$ in a suitable domain. The function $P_\lambda f
$ is an eigenfunction of the Laplacian $\textbf{L}$ of the group $G$, satisfying
$\textbf{L} P_\lambda f(x)= - (1+\lambda^2) P_\lambda f(x)$ and the
transform $f \mapsto P_\lambda f$ is called the generalized spectral
projection. Strichartz in his series of papers \cite{Strichartz88},
\cite{Strichartz89}, \cite{Strichartz91}, \cite{Strichartz92}
initiated the project of reviewing Harmonic Analysis in terms of the
generalized spectral projection. Continuing this project Bray
\cite{Bray96} proved a spectral Paley-Wiener theorem for the
symmetric space $X= G/K$. Ionescu \cite{Ionescu00} characterized the
image  $P_\lambda \left(L^2(X)\right)$. Strichartz \cite{Strichartz89} determined the image of Euclidean Schwartz class functions under spectral projection.\\
The aim of this paper is to characterize the image of the $L^p$-Schwartz
space $\mc S^p(X)$ $(0<p \leq 2)$ under the transform $f
\mapsto P_\lambda f $. For each $f \in \mc S^p(X)$, $P_\lambda f$ is a function on $\mf a^*_\e \times X$ where $\mf a^*_\e =\left\{\lambda \in \C ~|~ |\Im \lambda| < \e= \left( \f{2}{p}-1 \right) \right\}$.

The characterization of the image of $\mc S^p(X)$ under the generalized spectral projection is divided into two parts.   In Section \ref{sec: Necessary Conditions}  we obtain some basic properties (necessary conditions) of the functions $P_\lambda f$ for $f \in \mc S^p(X)$.  Sufficient
conditions for a left-$K$-finite  function $f(\lambda , x)$ on  $\mathfrak a^*_\e \times X$  to be of the form $P_\lambda g (x) = f(\lambda, x)$
for some $g \in \mc S^p(X)$ are taken up in Section \ref{sec: Sufficient Conditions}.
 Finally, in Section \ref{sec: Inverse Paley-Wiener Theorem}, we shall characterize the image of certain subspace of $L^2(X)$ under the above mentioned transform in the light of the inverse Paley-Wiener theorem due to Thangavelu \cite{Veluma07}.
 For this section we refer \cite{Krotz-Gindikin-I06, Krotz-Gindikin-II06, Krotz-Gindikin02, Krotz-Stanton-04, Krotz-Stanton-05}
 and \cite{Veluma07}.

\section{\textbf{Preliminaries}}
\label{sec: Prelminaries}
In this section we shall briefly recall some basic facts and results about noncompact Riemannian symmetric spaces realized as $X = G/K$ where $G$ be a connected noncompact semisimple Lie group with finite center and $K$ a maximal compact subgroup of $G$.  Let $\mf k $ be the Lie algebra of $K$.
 We fix an Iwasawa decomposition $G= KAN$ and let $\mf a$ be the Lie algebra of the abelian subgroup $A$. In this discussion we shall mainly concentrate on the `rank-one' Riemannian symmetric spaces so our associated semisimple Lie group $G$ will be of `real rank-one'. Hence $A$ and $\mf a$ both are of dimension one. Let $\mf a^*$ be the real dual of $\mf a$ and $\mf a^*_\C$ be its complexification. The Killing form induces a positive definite form $\langle \cdot, \cdot \rangle$ on $\mf a^* \times \mf a^*$, let the bilinear extension of this form on $\mf a^*_\C \times \mf a^*_\C$ be denoted by $\langle \cdot, \cdot \rangle_1$. In the rank-one case the restricted roots for the adjoint action of $\mf a$ on $\mf g$ are of the form $\pm \gamma$ and (possibly) $\pm 2\gamma$ with $\gamma \in \mf a^*$. We denote $m_\gamma $ for the multiplicity of the root $\gamma$. Let $\rho \in \mf a^*$ be the half sum of the positive roots. With a proper normalization the linear functional $\rho$ can be identified with the constant function $\rho(X)=1$. The same normalization also identifies $A$, $\mf a$, $\mf a^*$ with $\R$ and $\mf a^*_\C$ with $\C$. The Weyl group in our case $W =\{+1, -1 \}$ acts on $A$, $\mf a$, $\mf a^*$ and $\mf a^*_\C$ simply by multiplication. We denote $\mf a^+$ a cone in $\mf a$ called the positive Weyl chamber.  Let $\mf a^{* +} \subset \mf a^* $  be the cone dual to  $\mf a^+$.   In the rank-one case both  $\mf a^+$ and $\mf a^{*+}$ essentially correspond to the set $\R^+$ of all positive numbers.  Let $\overline{\mf a^+}$ and $\overline{\mf a^{*+}}$ respectively denote the closures of the cones $\mf a^+$ and $\mf a^{*+}$ also let us denote $\overline{A^+} = \exp{\overline{\mf a^+}}$.  The group elements of $A$ will now be denoted by $a_t$ where $t \in \R$ and $\exp t = a_t$.
 Let $H: g= ka_t n \mapsto H(g)=t \in \mf a$ be the Iwasawa-$\mf a$-projection of $G$ in $\mf a$ for the $KAN$ decomposition of the group. \\
 The Cartan decomposition gives $G= K \overline{A^+} K$.
  It induces a diffeomorphism from $K/M \times A^+ \times K$ (or $K \times A^+ \times M\setminus K$)
  onto an open dense subset of $G$ where $M$ is the centralizer of $A$ in $K$ ($M$ also normalizes $N$).
  Let $x^+$ be the $\overline{\mf a^+}$ projection of $x\in G$ for the Cartan decomposition
  $x = k_1 (\exp{x^+}) k_2$ and we denote $|x| = \| x^+\|$. For all $x \in G$ the
   Iwasawa-$\mf a$-projection $H(x)$ and the quantity $|x|$ are related by the inequality:
\begin{equation}
\label{eq:pri:coset distance and A-projection}
\|H(x)\| \leq c |x|, ~~ x \in G, \hbox{~where $c>0$ is a fixed constant}.
\end{equation}
We also note that in the symmetric space $X=G/K$, $|x|$ is the Riemannian distance of $xK$ from the coset $eK$, $e$ being the identity element of $G$.\\
The Haar measure corresponding to the Iwasawa-$KAN$ decomposition is given by
\begin{equation}
\label{eq:pri:Haar measure for KAN}
\int_G f(x) dx = const. \int_K dk  \int_{\mf a^+} e^{2t} d t \int_N f(k a_t n) dn ,
\end{equation}
where  the $const$ stands for a normalizing constant.
In the case of the Cartan decomposition the Haar measure on $G$ is given by
\begin{equation}
\label{eq:pri:Haar measure for KAK}
\int_G f(x) dx = const. \int_K dk_1  \int_{\mf a^+} \Delta(t)  dt \int_K f(k_1 a_t k_2) dk_2 ,
\end{equation}
where the density function $\Delta(t)$ has the estimate $\Delta(t)=O(e^{2t})$.
  A function $f$ will be called right-$K$-invariant  if  it  satisfies $f(xk)=f(x)$ for all $x \in G$ and $k \in K$.
A function on the symmetric space $X=G/K$ can also be considered as a right-$K$-invariant function on the group $G$.
We denote $\mc C^\infty(G)$ for the set of all smooth functions on $G$.
 Let $\mc U(\mf g)$ be
the `universal enveloping algebra' over $\mf g$ and $\Omega$ be the Casimir element of $\mc U(\mf g)$.
The action of the Laplace-Beltrami operator $\textbf{L}$ on $X$ is defined by the action
 of Casimir operator $\Omega$: $ \textbf{L} f(xK)= f(x; \Omega)$ for all $x \in G$.
  In our discussion the following eigenspaces of the operator $\textbf{L}$
\begin{equation}
\label{lambda eigenspace of L}
\mc E_\lambda(X)= \{g \in \mc C^\infty(X)~|~ \textbf{L}g(x)= -(\lambda^2 + 1) g(x) \}, \mbox{~for each~} \lambda \in \C
\end{equation}
will be of main interest.
For each $\lambda \in \C$, let $\varphi_\lambda$ be the elementary spherical function associated with $\lambda$. We recall that $\varphi_\lambda$ ($\lambda \in \C$) is given by the following integral representations \cite{Gangolli88}:
\begin{equation}
\label{eq:pri:elementary spherical functions def-1}
\varphi_\lambda(x)= \int_K e^{-(i \lambda+ 1)H(x^{-1}k)} dk=\int_K e^{(i \lambda - 1)H(x k )} dk.
\end{equation}
We recollect some of the very basic properties of the elementary spherical functions, which will be used throughout.
\begin{Proposition}
\label{pro:pri:properties of the elementary sp. funct.}
\begin{enumerate}
  \item The expression $\varphi_\lambda(x)$ is a bi-$K$-invariant $\mc C^\infty$ function in the $x$ variable and it is a $W$-invariant holomorphic function in $\lambda \in \C$.
  \item For each $\lambda \in \C$, $x \mapsto \varphi_\lambda(x)$ is a joint eigenfunction of all the $G$-invariant differential operators on $G/K$; in particular for the Laplace-Beltrami operator we have:
      \begin{equation}
      \label{eq:pri:ele. sp.funct. as eigen funct. of Laplacian}
      \textbf{L} \varphi_\lambda(\cdot) = -(\lambda^2 + 1) \varphi_\lambda(\cdot), \hspace{.3in}  \lambda \in \C.
      \end{equation}
      \item For each $\lambda \in \C$ and $x, y \in G$, the following property is referred to as the `symmetric property of the elementary spherical functions'
          \begin{equation}
          \label{eq:pri:symmetric property of the ele. sp. funct.}
          \varphi_{\lambda}(x^{-1}y) = \int_K e^{-(i \lambda+1)H(y^{-1}k^{-1})} e^{(i \lambda - 1)H(x^{-1}k^{-1})} dk.
          \end{equation}
          \item For any given $\textbf{D}, \textbf{E} \in \mc U(\mf g)$, there exists a constant $c>0$ such that \vspace{-.03in}
              \begin{equation}
              \label{eq:pri:estimate od sp.func. under the action of U g}
              |\varphi_\lambda(\textbf{D}; x; \textbf{E})| \leq c (|\lambda|+1)^{deg \textbf{E} + deg \textbf{D}} \varphi_{i \Im {\lambda}}(x) \hspace{.2in} \mbox{for all~} x \in G ,~\lambda \in \C.
              \end{equation}
          \item Given any polynomial $P$ in the algebra $\mathbf{S}(\mf a)$ of symmetric polynomials on $\mf a^*$, there exists a positive constant $c$ such that:
              \begin{equation}
              \label{eq:pri:estimate od sp.func. under the action of S a*}
              \lt|P\lt(\f{\partial}{\partial \lambda}\rt) \varphi_\lambda(x) \rt| \leq c (1 + |x|)^{deg P} \varphi_{i \Im{\lambda}}(x), \hspace{.2in}  x \in G.
              \end{equation}
          \item For all $t$ and $\lambda$ in $\overline{\R^+}$ we have:
              \begin{equation}
              \label{eq:pri:estimate od sp.func. on the imaginary axis}
              0 < \varphi_{- i \lambda} (a_t) \leq e^{\lambda t} \varphi_0(a_t).
              \end{equation}
              \item For all $x \in G$, we have $0 < \varphi_0(x)=\varphi_0(x^{-1}) \leq 1$;
              \item  For all $t \in \overline{\R^+}$, we have the following two-side estimate of $\varphi_0$:
                  \begin{equation}
                  \label{eq:pri:apraori estimate of phi 0}
                  e^{-t} \leq \varphi_0(a_t) \leq c (1 + t)^{a} e^{- t},
                  \end{equation}
                  where $c, a>0$ are group dependent constants;
\end{enumerate}
\end{Proposition}
Property (i) is a very basic fact which follows from the definition. For a proof one can see \cite[Ch. 4]{Gangolli88}. Property (ii) was proved by Helgason \cite{Helgason-gga}. For (iii) we refer to
\cite[Ch. III, Theorem 1.1]{Helgason-gas}. The estimates (iv), (v), and (vi) follows from the results in \cite[Sec. 4.6]{Gangolli88}. For a direct and a simple proof of (iv) and (v) one can see \cite[Proposition 3]{Anker91}. The estimate (vii) of $\varphi_0$ is due to Harish-Chandra. A proof of this can be found in
\cite[Theorem 4.6.4, Theorem 4.6.5]{Gangolli88}. We should note that a sharper
two-sided estimate of $\varphi_0$ is given by Anker  \cite{Anker87}.
\vspace{.1in}

Let $\delta$ be an unitary irreducible representation of K i.e
$\delta \in \widehat{K}$ with $V_\delta$ (a finite dimensional
vector space) the representation space. Let $\chi_\delta$ stand for
the character of the representation $\delta $. Let $ V_\delta^M$ be
the subspace of $V_\delta$ fixed under $\delta|_M$; i.e $ V_\delta^M
= \{ v \in V_\delta ~| ~\delta(m)v = v ~ \forall m \in M\}$. Kostant
\cite{Kostant69} proved that for a rank-one group the dimension of $V_\delta^M
$ is 0 or 1. Let $\widehat{K}_M$ be the set of all equivalence
classes of irreducible unitary representation $\delta$ of K for
which $V_\delta^M \neq \{0\}$. For our result we choose $\delta \in
\widehat{K}_M$.


As, $\delta \in \widehat{K}_M$, $\delta(k)$ is a unitary matrix of
order $d_\delta$. So  $\|\delta(k)\|_{\mathbf{2}} = d_\delta^{\frac{1}{2}}$
where, $\|\cdot\|_{\mathbf{2}}$ denotes the Hilbert Schmidt norm.
One can associate a norm $|\delta|$ to each unitary irreducible
representation $\delta$ of $K$ (for explicit constriction of this norm we refer \cite{Eguchi76}).
From
Weyl's dimension formula we can choose an $r \in \mathbb Z^+$ and a
positive constant $c$ independent of $\delta$ such that
\begin{eqnarray}\label{delta-estimate}
\|\delta (k)\|_{\mathbf{2}} &\leq &c~(1+|\delta|)^r
\end{eqnarray}
for all $k \in K$. Thus
$d_\delta \leq c' (1+|\delta|)^{2r}$
clearly, again $c'$ is independent of the chosen $\delta$.\\
For any $f \in \mc C^\infty(X)$ we put:
\begin{equation}
\label{eq:pri:matrix valud delta projection}
f^\delta(x) = d_\delta \int_K f(kx) \delta(k^{-1}) dk.
\end{equation}
Clearly, $f^\delta$ is a $\mc C^\infty$ map from $X$ to $Hom(V_\delta, V_\delta)$ satisfying
\begin{equation}
\label{eq:pri:property of f delta}
f^\delta(kx) = \delta(k) f^\delta(x), ~~\mbox{for all~} x \in X, k \in K.
\end{equation}
Any function satisfying the property (\ref{eq:pri:property of f delta}) will be referred to as (a $d_\delta\times d_\delta$ matrix valued) left $\delta$-type function. For any function space $\mc E(X) \subseteq \mc C^\infty(X)$, we write $\mc E^\delta(X)= \{f^\delta~|~f \in \mc E(X) \}$.  We shall denote by $\check{\delta}$ the \emph{contragradient} representation of the representation $\delta \in \what{K}_M$.
A function $f$ will be called a scalar valued left-$\delta$-type function if $ f = d_{\delta} \chi_{\delta} \ast f$, where the operation $\ast$ is the convolution over $K$. For any class of scalar valued  functions $\mc G(X)$ we shall denote
$\mc G(\delta, X) = \{g \in \mc G(X)~|~g = d_{\delta} \chi_{\delta} \ast g \}$.
Throughout our discussion we fix the notation $\mc D(X)$ for the subclass of functions in $\mc C^\infty(X)$ which are of compact support. The following theorem, due to Helgason,  identifies the two classes $\mc D^\delta(X)$ and $\mc D(\check{\delta}, X)$ corresponding to each $\delta \in \what{K}_M $.
\begin{Theorem}
\label{theo:pri:left delta left delta check identification}
\emph{[Helgason \cite[Ch.III, Proposition 5.10]{Helgason-gas}]}\\
The map $\mc Q: f \mapsto g$,  $g(x)=tr\lt(f(x)\rt)$  $(x \in X)$ is a homeomorphism from $\mc D^\delta(X)$ onto $\mc D(\check{\delta}, X)$ and its inverse is given by $g \mapsto f=g^\delta$.
\end{Theorem}
A consequence of the  `Peter-Weyl theorem' can be stated  \cite[Ch.IV, Corollary 3.4]{Helgason-gga} in the form that  any $f \in \mc C^\infty(X)$ has the decomposition
\begin{equation}
\label{eq:pri:Peter-Weyl decomposition}
f(x)= \sum_{\delta \in \what{K}_M} tr(f^\delta(x)),
\end{equation}
where the convergence is in the sense of uniform convergence on compacta.
A function $f \in \mc C^\infty(X)$ is said to be left-$K$-finite if there exists a finite subset $\Gamma(f) \subset \what{K}_M$ (depending on the function $f$)
such that  $tr(f^\gamma) = 0$ for all $\gamma \in \what{K}_M \setminus \Gamma(f) $. For any class $\mf H(X) \subseteq \mc C^\infty(X) $ of functions we shall denote by $\mf H(X)_K$ the subclass consisting of the  left-$K$-finite functions. Let $\Gamma$ be a fixed subset (finite or infinite) of $\what{K}_M$. Then we shall use the notation $\mf H(\Gamma ; X)$ for the  subclass of $\mf H(X)$
\begin{equation}
\label{eq:pri:finite left K type defn}
\mf H(\Gamma ; X)= \{g \in \mf H(X)~|~ g^\delta = 0, \mbox{~for all~} \delta \in \what{K}_M \setminus \Gamma \}.
\end{equation}

 The $p$th Schwartz space $\mc S^p(X) $ ($1< p \leq 2$) on $X$ is the class
of $\mc C^\infty $ functions $f$ on $X$ with the decay condition: for each
$D, E \in \mathcal U(\mathfrak g)$ and $n \in \Z^+ \cup \{0\}$
\begin{equation}\label{Schwartz space on X}
\mu_{D, E, n}(f) :=~ \sup_{x \in X} |f(D;x;E)| (1+|x|)^n
\varphi_0^{-\frac{2}{p}}(x) < +\infty.
\end{equation}
The quantities $\mu_{D, E, n}(f)$ gives a countable family of
seminorms on the space $\mc S^p(X)$ and the topology induced by this countable family  makes $\mc S^p(X)$  a Fr\'{e}chet space.

Let us denote $\mc S^p(\check{\delta}, X) = \{ f \in \mc S^p(X)~| ~ f = d_\delta~\chi_{\check{\delta}} \ast f ~\}$ and  ${\mc S^{p}}^\delta(X) = \{ f^\delta~ |~ f \in \mc S^p(X) \} $. Being a closed subspace of $\mc S^p(X)$ the space $\mc S^p(\check{\delta}, X)$ is also a Fr\'{e}chet space with the topology induced from $\mc S^p(X)$. The $Hom(V_\delta, V_\delta)$ valued function space ${\mc S^{p}}^\delta(X)$ is also a Fr\'{e}chet space with respect to the topology induced by the countable family of seminorms: for $D, E \in \mc U(\mf g)$ and $n \in \Z^+ \cup \{ 0\}$
\begin{equation}
\label{seminorms on the operator valued SW space}
\mu_{D, E, n}(f^\delta):= \sup_{x \in X} \|f^\delta(D;x;E)\|_{\mathbf{2}} (1+|x|)^n
\varphi_0^{-\frac{2}{p}}(x) < +\infty.
\end{equation}
Note that for the sake of simplicity we keep the same notation for the seminorms for both the spaces $\mc S^p(X)$ and ${\mc S^{p}}^\delta(X)$.
Clearly the spaces $\mc D(\check{\delta}, X)$ and $\mc D^\delta(X)$ respectively are dense subspaces of the Schwartz spaces $\mc S^p(\check{\delta}, X)$ and ${\mc S^{p}}^\delta(X)$ with the respective Schwartz space topologies.
 \begin{Remark}
 \label{Isomorphism between the scalar and op. valued SW spaces}
 The topological isomorphism described in Therem \ref{theo:pri:left delta left delta check identification} can be extended from the Schwartz space ${\mc S^{p}}^\delta(X)$ onto $\mc S^p(\check{\delta}, X)$.
 \end{Remark}
For any function $f \in \mc S^p(X)$, the Helgason
Fourier transform (HFT) $ \mc F f$ is defined by
\begin{equation}\label{Helgason Fourier transform}
\mc F f(\lambda, kM) = \int_X f(x) ~e^{(i\lambda -1)H
(x^{-1}k)} dx.
\end{equation}
For $f \in \mc S^p(X)$ the function $
\mc F f$ is defined on the domain $\mf a^*_\e \times K/M$.  Now let $\mc S(\mf a^*_\e \times K/M)$ denote the space consisting of the $\mc C^\infty$ functions $h$ on $\mf a^* \times K/M$ which satisfy the following conditions:
\begin{enumerate}
  \item For each fixed $k \in K$ the function $\lambda \mapsto h(\lambda, kM )$ extends to $Int \mf a^*_\e$ as a holomorphic function and it extends as a continuous function to the closed strip $\mf a^*_\e$.
  \item For any $\lambda \in Int \mf a^*_\e$ and $x \in G$ we have $\check{h}(-\lambda, x)= \check{h}(\lambda, x)$ where $$\check{h}(\lambda,x)=\int_K h(\lambda, \texttt{K}(xk)M)e^{(i\lambda-1)H(xk)} dk, $$
      where $\texttt{K}(xk)$ is the $K$ part of $xk\in G$ in the Iwasawa $KAN$ decomposition.
  \item For each $m, n \in \Z^+$ and $P \in \mathbf{S}(\mf a)$
  \begin{equation}\label{SW semi norm for image of HFT}
  \sup_{\lambda \in Int \mf a^*_\e; k \in K/M} \left|P\left(\f{d}{d\lambda}\right) h(\lambda, k: \omega_{\mf k}^m)\right| (1 + |\lambda|)^n < +\infty
  \end{equation}
  where $\omega_{\mf k}$ is the Casimir element of $\mf k$.
 \end{enumerate}
The space $\mc S(\mf a^*_\e \times K/M)$ becomes a Fr\'{e}chet space with the topology induced by the countable family of seminorms (\ref{SW semi norm for image of HFT}). It can be shown that the HFT is a continuous mapping of $\mc S^p(X)$  into  $\mc S(\mf a^*_\e \times K/M)$ \cite[Theorem 3.1]{Eguchi76}.
Let us fix the notation $\mc F f(\lambda, kM) =
\mc F f(\lambda, k)$. The Inversion formula for HFT for $f \in
\mc S^p(X)$ is given by

\be \label{HFT Inversion} f(x) = \frac{1}{2} \int_{\mathfrak
a^*} \int_{K}  \mc F f(\lambda, k)~e^{-(i\lambda +1)H(x^{-1}k)}
|\textbf{c}(\lambda)|^{-2} d \lambda~ dk. \ee

Here $\textbf{c}(\lambda)$ is the
Harish-Chandra $\textbf{c}$-function which is completely known (see,
\cite[ Chap. IV]{Helgason-gga}; \cite[ Sect. 4.7]{Gangolli88}). For
our purpose we shall only need the following simple estimate \cite{Anker92} : for
constants $c,b
> 0$
\be \label{estimate of  c funct.} |\textbf{c}(\lambda)|^{-2}
\leq~c(|\lambda| + 1)^b ~~\mbox{ for all $\lambda \in \mathfrak a^*$}.
\ee

Let $f \in \mc S^p(X)$ then $f^\delta \in \mc S^p_\delta(X)$. The HFT of
$f^\delta$  is defined similarly as (\ref{Helgason Fourier
transform}) by
\begin{equation}
\mc F(f^\delta)(\lambda, kM) = \int_G f^\delta(x) e^{(i
\lambda-1)H(x^{-1}k)} dx,
\end{equation}
here the integration is done on each matrix entry. It can be shown
that $$\mc F(f^\delta)(\lambda, kM) = \delta(k)
\mc F(f^\delta)(\lambda, eM)=\delta(k)
\widetilde{f^\delta}(\lambda),$$ here we are denoting
$\mc F(f^\delta)(\lambda, eM)$ by
$\widetilde{f^\delta}(\lambda)$. The function
$\widetilde{f^\delta}(\lambda)$ can be represented by the following
(for details see \cite{Jana}):
\begin{eqnarray}\label{def delta spherical}
\widetilde{f^\delta}(\lambda) = d_\delta \int_X
trf^\delta(x)\Phi_{\overline{\lambda}, \delta}^*(x) dx ,
\end{eqnarray}
where for each $\delta \in \what{K}_M$ and $\lambda \in \C$, the function
\begin{equation}
\label{eq:pri:gen. sp. funct.}
\Phi_{\lambda, \delta}(x) = \int_K e^{-(i \lambda+ 1)H(x^{-1}k)} \delta(k) dk, \hspace{.2in}  x\in G,
\end{equation}
is called the `generalized spherical function' of class $\delta$. For each $x \in G$, $\Phi_{\lambda, \delta}(x)$  is an operator in $Hom(V_\delta, V_\delta)$. Taking point-wise adjoints leads to the expression
\begin{equation}
\label{eq:pri:adjoint of gen. sp. funct.}
\Phi_{\overline{\lambda}, \delta}^*(x):=\Phi_{\overline{\lambda}, \delta}(x)^* = \int_K e^{(i \lambda -1)H(x^{-1}k)} \delta(k^{-1}) dk, \hspace{.2in}  x\in G.
\end{equation}
We note that from the  Iwasawa decomposition, if $x \in G$ and $\tau\in K$, $H(\tau x)= H(x)$. Hence, the expressions (\ref{eq:pri:gen. sp. funct.}) and (\ref{eq:pri:adjoint of gen. sp. funct.}) show that both $\Phi_{\lambda, \delta}$ and $\Phi_{\overline{\lambda}, \delta}^*$ can be considered as functions on the  space $X=G/K$. The transform $f^\delta \mapsto \widetilde{f^\delta}$ given by the integral (\ref{def delta spherical}) will be referred as the $\delta$-spherical transform.
We list out some basic properties of the generalized
spherical functions $\Phi_{\lambda, \delta}$ in the following remark
\begin{Remark}
\label{properties of the gen-sp functions}

\begin{enumerate}
\item
The following can  easily be checked from the  integral representations given in
(\ref{eq:pri:gen. sp. funct.}) and (\ref{eq:pri:adjoint of gen. sp. funct.})
of the generalized spherical function and it's adjoint
\begin{itemize}
\item
for all $k \in K$, $\Phi_{{\lambda}, \delta}(kx) =
\delta(k) \Phi_{{\lambda}, \delta}(x)$  and
$\Phi^*_{\overline{\lambda}, \delta}(kx) = \Phi^*_{\overline{\lambda}, \delta}(x)
\delta(k^{-1})$.
\item
Let $v \in V_\delta$ and $ m \in M$ then $\delta(m)
\left( \Phi^*_{\overline{\lambda}, \delta}(x) v\right)=
\Phi^*_{\overline{\lambda}, \delta}(x) v$.
\end{itemize}
This shows that $\Phi^*_{\overline{\lambda}, \delta} (\cdot) $ is a $ Hom
(V_\delta,V_\delta^M)$ valued function on $X$. Hence, the $\delta$-spherical transform
$\widetilde{f^\delta}(\cdot)$ as defined in (\ref{def delta spherical}) is a $ Hom
(V_\delta,V_\delta^M)$   valued function on $\C$.
\item
For each $\lambda \in \C$ and $\delta \in \widehat{K}_M$, the functions
$x \mapsto \Phi_{\lambda , \delta}(x)$ is a joint eigenfunction  of the algebra $\mathbf{D}(X)$ of all $G$-invariant differential operators on $X$.
\item
For each fixed $\delta \in \widehat{K}_M$ there exists a polynomial
 \cite[ Theorem 5.15, ch-III, $\S$5, p-289 ]{Helgason-gas}
$Q_\delta(1-i\lambda)$ of the complex variable $i \lambda$
such that: for all $\lambda \in \C$
\begin{eqnarray} \label{Intro- of kostant poly1}
Q_\delta(1-i \lambda) \Phi_{\lambda, \delta} (\cdot) &=&
Q_\delta(1+i \lambda) \Phi_{-\lambda, \delta} (\cdot)\\\label{Intro- of kostant poly2}
\left[Q_\delta(1-i \lambda) \right]^{-1} \Phi^*_{\overline{\lambda},
\delta} (\cdot)
&=&  \left[Q_\delta(1+i \lambda) \right]^{-1} \Phi^*_{-\overline{\lambda}, \delta} (\cdot)
\end{eqnarray}
Both sides of the above relations are holomorphic for all $\lambda
\in \C$.
\end{enumerate}
\end{Remark}
The polynomials $Q_\delta$ are called the Kostant's polynomials. For a rank-one group the polynomial $Q_\delta(1+ i \lambda)$ has the following
representation interms of the Gamma functions \cite[Theorem 11.2, Ch. III, $\S$11]{Helgason-gas}
\begin{equation}
\label{Structure of the Kostant's polynomial}
Q_\delta(1+i\lambda)=
\left(\frac{1}{2}(\alpha +\beta+1+i \lambda) \right)_{\frac{r+s}{2}}
\left(\frac{1}{2}(\alpha - \beta +1 +i \lambda)
\right)_{\frac{r-s}{2}}
\end{equation}
where, $(z)_m = \frac{\Gamma(z+m)}{\Gamma(z)}$. Two group dependent
constants $\alpha, ~\beta$ are given by $\alpha =
\frac{1}{2}(m_\gamma +m_{2 \gamma} -1)$, $\beta =\frac{1}{2}(m_{2
\gamma} -1) $. The pair of integers $(r,s)$ is the parameterization
of the representation $\delta \in \widehat{K}_M$. Clearly
$Q_\delta(1+i \lambda)$ is a polynomial in $i \lambda$ of order
$2r$. Helgason \cite[Ch. III, $\S$ 11]{Helgason-gas}  further showed that
the polynomial $Q_\delta(1+i \lambda)$ has no zero in the interior of the
strip $\mathfrak a^*_\e:= \{\lambda \in \mf a^*_\C ~:~ |\Im \lambda| \leq
\e \} $.

We now define a function space in the Fourier domain which is a prospective candidate for the image of ${\mc S^{p}}^\delta(X)$ under the $\delta$-spherical transform.
\begin{Definition}
\label{def:ch2:delta SW space in the image side}
We denote $\mc S_\delta(\mathfrak a^*_\e)$ for the space of all $Hom(V_\delta, V_\delta)$ valued functions $\psi$ on the complex strip $\mathfrak a^*_\e$ with the properties:
\begin{enumerate}
  \item For each $\lambda \in \mf a^*_\e$,  $\psi(\lambda)$ maps $V_\delta$ to $V_\delta^M$.
    \item Each $\psi $ is holomorphic in the interior of the strip $\mf a^*_\e$ and  extends as a continuous function to the closed strip.
  \item $\psi$ satisfies the identity
      \begin{equation}
      \label{eq:ch2:W action on the image SW space}
      Q_\delta(1-i \lambda)^{-1} \psi(\lambda) = Q_\delta(1+i \lambda)^{-1} \psi(- \lambda),\hspace{.2in} \lambda \in \mf a^*_\e,
      \end{equation}
      where $ Q_\delta(1+i \lambda)$ is the polynomial (\ref{Structure of the Kostant's polynomial}).
      \item For each $P \in \mathbf{S}(\mf a)$ and for each integer $t \geq 0$ we have:
          \begin{equation}
          \label{eq:ch2:SW decay on the image side of delta sp. trans}
          \tau_{P, t}(\psi) = \sup_{\lambda \in Int \mf a^*_\e} \lt\|P\lt(\f{d}{d\lambda}\rt) \psi(\lambda)\rt\|_{\mathbf{2}} (1 + |\lambda|)^t < + \infty.
          \end{equation}
\end{enumerate}
\end{Definition}
 We have already mentioned that $dim V_\delta^M = 1$, so for each $\psi \in \mc S_\delta(\mathfrak a^*_\e)$ and $\lambda \in \mf a^*_\e$, with a convenient choice of basis,  $\psi(\lambda)$ is a $d_\delta \times d_\delta$ matrix with all the rows except the first one being identically zero.
It can be shown that the space $\mc S_\delta(\mf a^*_\e)$ is a Fr\'{e}chet space with the topology induced by the countable family of seminorms $\{ \tau_{P, t}\}$.
We shall be using the following topological characterization of the image of the Schwartz space ${\mc S^{p}}^\delta(X)$ under the $\delta$-spherical transform.
\begin{Theorem}
\label{theo:ch2:main}
 For $0 < p \leq 2$ and $\e = (2/p-1)$ the $\delta$-spherical transform $f \mapsto \widetilde{f}$ is a topological isomorphism between the spaces ${\mc S^{p}}^\delta(X)$ and $\mc S_\delta(\mathfrak a^*_\e)$.
\end{Theorem}
This is a part of the result proved by Eguchi and Kawata \cite{Eguchi76}.
A proof of this theorem avoiding the complicated asymptotics of the generalized
 spherical functions can be found in \cite{Jana}.

\section{\textbf{Necessary Conditions}}
\label{sec: Necessary Conditions}

\setcounter{equation}{0}
In this section we start with the $L^p$-Schwartz space  $\mc S^p(X)$ with $0 < p \leq 2$. For $f \in \mc S^p(X)$ we
define $P_\lambda f(x)= (f \ast \varphi_\lambda)(x)$, for suitable $\lambda \in \C
$.
We get  an alternative expression for  the spectral projection $P_\lambda f(\cdot)$ in terms of
the Helgason Fourier transform $\mc F f$ of the function $f \in \mc S^p(X)$.
Beginning with $P_\lambda f(x)= \int_{G} f(y) \varphi_\lambda(y^{-1}x) dy$,
we use the standard symmetric property (\ref{eq:pri:symmetric property of the ele. sp. funct.})
of the elementary spherical functions and the Fubini's theorem to write:
  \begin{align}\label{P lambda}
P_\lambda f(x) &= \int_K \lt\{ \int_G f(y)  e^{(i \lambda - 1)H(y^{-1}k^{-1})} dy \rt\} e^{-(i \lambda +1)H(x^{-1}k^{-1})} dk \nonumber\\
&= \int_K \mc F f(\lambda, k^{-1}) e^{-(i \lambda +1)H(x^{-1}k^{-1})} dk \nonumber\\
&= \int_{K} \mc F f(\lambda, k) e^{-(i \lambda+ 1) H(x^{-1}k)} dk.
\end{align}
We have already mentioned  that for any $f \in \mc S^p(X)$ the Helgason
Fourier transform $\mc F f$  is defined on the domain $\mf a^*_\e \times K/M$.
 Hence, (\ref{P lambda}) implies that for
  each $f \in \mc S^p(X)$ the function  $(\lambda, x)\mapsto P_\lambda f(x)$
   is defined on $\mf a^*_\e \times X$.
We use the notation $\mf e_{\lambda, k}(x)$ for $e^{-(i \lambda+ 1) H(x^{-1}k)} $, which is the kernel of the integral in the definition (\ref{P lambda}). As we have already mentioned that the Iwasawa decomposition $KAN$ is diffeomorphic to $G$, so, for each $k \in K$, the Iwasawa-$\mf a$-projection $x \mapsto H(x^{-1}k)$ is a $\mc C^\infty $ map on $G$. Thus $\mf e_{\lambda, k} \in \mc C^\infty(X)$ for each $\lambda \in \C$ and $k \in K$. Hence, from (\ref{P lambda}) one can conclude that for each $\lambda \in \mf a^*_\e$ and for each $f \in \mc S^p(X)$, $P_\lambda f  \in \mc C^\infty(X)$. Furthermore the kernel $\mf e_{\lambda, k} $ is a joint eigenfunction of the algebra $\mathbf{D}(X)$. In particular for the Laplace-Beltrami operator $\textbf{L}$ the eigenvalue for $\mf e_{\lambda, k}$ we have:
\begin{equation}
\label{eq:kernel as eigen funct of laplacian}
\textbf{L} \mf e_{\lambda, k}(x) = -(1+ \lambda^2) \mf e_{\lambda, k}(x),
\end{equation}
for each $\lambda \in \C$ and $k \in K$ \cite{Gangolli88}. According to our notation $\mf e_{\lambda, k} \in \mc E_\lambda(X)$ for each $\lambda \in \C$. Note that the integral  in the definition (\ref{P lambda}) of $P_\lambda f$ is over a compact set. Therefore, it follows easily that for each $\lambda \in \mf a^*_\e$ and $f \in \mc S^p(X)$,   $P_\lambda f \in \mc E_\lambda(X)$.

To prove other characteristic properties of the function space $P_\lambda(\mc S^p(X))$ ($\lambda \in \mf a^*_\e$). We shall mainly use the continuity
 of the $\delta$-spherical transform \cite[Lemma 4.2]{Jana} which is a part of Theorem \ref{theo:ch2:main}.

For each $f \in \mc S^p(X)$ and for each $\lambda \in \mf a^*_\e$,  $P_\lambda f \in \mc C^\infty(X)$,
we define its matrix valued left $\delta$-projection $(P_\lambda f)^\delta$ by
\be
\label{eq:delta projection of P lambda}
(P_\lambda f)^\delta(x)= d_\delta \int_K P_\lambda f(k x) \delta(k^{-1}) dk.
\ee

It is clear that for each $\delta \in \what{K}_M$, $(P_\lambda f)^\delta \in \mathcal E_\lambda^\delta(X) $.
Now $(P_\lambda f)^\delta$ satisfies $(P_\lambda f)^\delta (k x) = \delta(k)( P_\lambda f)^\delta (x)$ ($ k \in K, x \in X $). Hence $tr(P_\lambda f)^\delta $ is a left $\check{\delta}$-type scalar valued  function and hence $tr (P_\lambda f)^\delta \in \mc E_\lambda(\check{\delta}, X)$.\\
The following proposition relates the projection $(P_\lambda f)^\delta$  with the generalized spherical function (\ref{eq:pri:gen. sp. funct.}). This structure will be very useful for estimating the decay of the function $P_\lambda f$ for each $f \in \mc S^p(X)$ ($0 < p \leq 2$).

\begin{Proposition}
\label{delta projection of P-lambda}
Let $f \in L^p(X)$ if $1 < p \leq 2$ and $f \in \mc S^p(X)$ if $0< p \leq 1$. Then  ~~~~$(P_\lambda f)^\delta(x)
= P_\lambda (f^\delta)(x) = \Phi_{\lambda, \delta}(x)~
\widetilde{f^\delta}(\lambda),\mbox{~for~}x \in X \mbox{~and~} ~\lambda \in Int \mf a^*_\e,$  where $\widetilde{f^\delta}$ is the
$\delta$-spherical transform of $f^\delta$ as defined in (\ref{eq:pri:matrix valud delta projection}).
\end{Proposition}
\begin{proof}
The existence of $P_\lambda f$ needs a proof in case of $f \in L^p(X)$, $1 < p \leq 2$.
It is a consequence of the estimates (\ref{eq:pri:estimate od sp.func. under the action of U g}) and (\ref{eq:pri:estimate od sp.func. on the imaginary axis}) of the function $\varphi_\lambda$, that for $\lambda \in \mf a^*_\e$, $\varphi_\lambda \in L^q(X)$, where $\f{1}{p}+\f{1}{q}=1$. Further, it can be shown that for each compact set $U\subset Int \mf a^*_\e$, there exists a $g \in L^q(X)$ with $g \geq 0$ such that $|\varphi_\lambda(x)| < g(x)$, for $\lambda \in U$ and $x \in X$. Thus by H\"{o}lder's inequality $f \ast \varphi_\lambda(x)$ exists for all $\lambda \in U$. Moreover, the uniform domination of the $\varphi_\lambda$ means that $\lambda \mapsto \varphi_\lambda$ is a continuous map of $U$ to $L^q(X)$. H\"{o}lder's inequality will then make $f \ast \varphi_\lambda(x)$ continuous in $\lambda\in U$. The compact set $U$ being arbitrary we get the existence and continuity in both the variables on $Int \mf a^*_\e \times X$.
From (\ref{eq:delta projection of P lambda}) we now have:
\begin{eqnarray}
(P_\lambda f)^\delta(x) &=& d_\delta\int_K P_\lambda f(k x)~ \delta(k^{-1})
~dk \nonumber \\
&=& d_\delta \int_K \int_G f(y) \varphi_\lambda (y^{-1}kx) ~dy~
\delta(k^{-1})
~dk \nonumber \\
&=& d_\delta \int_K \int_G f(kz) \varphi_\lambda (z^{-1}x)~ dz
~\delta(k^{-1})~
dk \nonumber \\
&=& d_\delta\int_G \varphi_\lambda (z^{-1}x) \int_K f(kz) ~\delta(k^{-1})
~dk  ~dz \nonumber\\
\label{No-1} &=& \int_G f^\delta (z) \varphi_\lambda (z^{-1}x)~ dz\\
&= & P_\lambda (f^\delta)(x) \nonumber.
\end{eqnarray}
Using the symmetric property of the elementary spherical function (\ref{eq:pri:symmetric property of the ele. sp. funct.}) the expression (\ref{No-1}) can be written as
\begin{equation}
(P_\lambda f)^\delta(x)= \int_G f^\delta(z) \int_K e^{-(i \lambda + 1)H(x^{-1}k)} e^{(i
\lambda -1)H(z^{-1}k)} dk~ dz.
\end{equation}
The repeated integral on the right-hand side converging absolutely,  we interchange the integrals to obtain
\begin{align} \label{structural form}
(P_\lambda f)^\delta(x)&=
\int_K  e^{-(i \lambda + 1)H(x^{-1}k)} \int_G f^\delta(z) e^{(i
\lambda -1)H(z^{-1}k)} dz~ dk \nonumber\\
&=  \int_K e^{-(i
\lambda + 1)H(x^{-1}k)} \mc F f^\delta(\lambda, k) dk \nonumber.
\intertext{We have already noticed that $\mc F f^\delta(\lambda, k) = \delta(k) \mc F f^\delta(\lambda, e) $ and also we have observed that $ \mc F f^\delta (\lambda, e) = \widetilde{f^\delta}(\lambda)$. Then}
(P_\lambda f)^\delta(x ) &= \bigg\{\int_K e^{-(i \lambda + 1)H(x^{-1}k)} \delta(k) dk \bigg\}
\widetilde{f^\delta}(\lambda)
= \Phi_{\lambda, \delta} (x) \widetilde{f^\delta}(\lambda).
\end{align}
\end{proof}
\begin{Remark} \label{imp}
As $\varphi_\lambda(x)=\varphi_{-\lambda}(x)$ for all
$\lambda \in \C $, so both the functions
$P_\lambda f$, $(P_\lambda f)^\delta $ and $tr (P_\lambda f)^\delta$
are even  in
the $\lambda$ variable.
\end{Remark}
To characterize $P_\lambda f$ for $f \in \mc S^p(X)$, we shall first concentrate on each of its $\delta$-projections $(P_\lambda f)^\delta$. The following proposition summarizes the properties of $(P_\lambda f)^\delta$ which will be useful to characterize  $P_\lambda (\mc S^p(X))$ ($\lambda \in \mf a^*_\e$).
\begin{Proposition}
\label{proposition-for decay of delta proj of Pf}
For $f \in \mc S^p(X)$, where $0 < p \leq 2$ and for each fixed $\delta \in \widehat{K}_M$
the operator valued
left $\delta$-projection
$(P_\lambda f)^\delta$ of $P_\lambda f$ ($ \lambda \in \mf a^*_\e$)
has the  properties:
\begin{enumerate}
\item For each $\lambda \in \mf a^*_\e$, the function $(P_\lambda f)^\delta \in \mc E_\lambda^\delta(X)$.
\item For each $x \in X$, $\lambda \mapsto (P_\lambda f)^\delta(x)$ is an even holomorphic function in the interior of the strip $\mf a^*_\e$ and it extends as an even continuous function on the closed strip. The map $\lambda \mapsto Q_\delta(1 - i \lambda)^{-1} (P_\lambda f)^\delta(x) $ is a holomorphic function on the open strip $Int \mf a^*_\e$. For $p= 2$, $(P_\lambda f)^\delta(x)$ is a real analytic function of $\lambda \in \R$.
\item  For each $\textbf{D}, \textbf{E} \in \mathcal U(\mathfrak
g)$, $m, n, s \in \Z^+ \cup \{0\}$ and for any real number $ r_p <\frac{2p-2}{p}$
we can find positive constants $c_i$ and positive integers $l, t$ such that
\begin{align}
& \sup_{x \in G, \lambda \in Int \mathfrak a^*_\e}\bigg \|\left(\frac{d}{d\lambda}
\right)^m (P_\lambda f^\delta)(\textbf{D};
x; \textbf{E}) \bigg\|_{\mathbf{2}} (1+|x|)^n (1+ |\lambda|)^s \varphi_0^{-r_p}(x) \nonumber\\
& \hspace{1.3in}\leq \sum_{i=0}^{m} c_i (1+|\delta|)^q\sup_{x \in X}
\|\textbf{L}^l f^\delta(x)\|_{\textbf{2}} (1+|x|)^t \varphi_0^{-\frac{2}{p}} (x).
\end{align}
\end{enumerate}
\end{Proposition}
\begin{proof}
Property (i) has already been discussed.  The property (ii) is a consequence of the expression (\ref{structural form}) and from the fact that the function $\lambda \mapsto \widetilde{f^\delta}(\lambda)$ satisfies property (ii) and (iii) of Definition \ref{def:ch2:delta SW space in the image side}.\\
(iii) ~~ Using the result of the Proposition \ref{delta projection of
P-lambda} we get
\begin{align}\label{breaking of P-lambda-delta as product}
& \lt\|\left(\frac{d}{d \lambda} \right)^m (P_\lambda f)^\delta(\textbf{D};
x; \textbf{E})\rt\|_{\textbf{2}} = \bigg\|\left(\frac{d}{d \lambda} \right)^m\{
\Phi_{\lambda, \delta}(\textbf{D}; x; \textbf{E})
\widetilde{f^\delta}(\lambda)\}\bigg\|_{\textbf{2}} \nonumber\\
&\hspace{1.5in}\leq \sum_{\ell=0}^t\bigg\|\left(\frac{d}{d \lambda} \right)^\ell
\Phi_{\lambda, \delta}(\textbf{D}; x; \textbf{E})\bigg\|_{\textbf{2}}~~\bigg\|\left(\frac{d}{d
\lambda} \right)^{m-\ell} \widetilde{f^\delta}(\lambda)\bigg\|_{\textbf{2}}
\end{align}
We shall use the following estimates for the various derivatives of the
matrix coefficients of the principal series representation  \cite[ $\S$17, Lemma 1]{Harish76}:
\begin{equation}
\label{bound of matrix coefficient of the principle series}
\bigg\|\left(\frac{d}{d \lambda} \right)^\ell \Phi_{\lambda, \delta}(\textbf{D};
x;\textbf{E})\bigg\|_{\textbf{2}} \leq c (1+|\delta|)^q (1+ |\lambda|)^q (1+ |x|)^u
\varphi_0(x) e^{|\Im{\lambda}| |x|}
\end{equation}
where $c >0$ is a constant (may depend on the derivatives chosen
but independent of $\delta \in \widehat{K}_M$ ), $q \in \mathbb Z^+$
depends on $\textbf{D}, \textbf{E} \in \mathcal U(\mathfrak g)$ and $u \in \mathbb
Z^+$ depends on the integer $\ell$. As $\lambda \in \mathfrak a^*_\e$ we can
replace $|\Im{\lambda}|$ by $\e= (\f{2}{p}-1)$ in (\ref{bound of matrix
coefficient of the principle series}). Now from (\ref{breaking of
P-lambda-delta as product}) and (\ref{bound of matrix coefficient of
the principle series}) we get:
\begin{align}\label{No. 2}
\bigg\|&\left(\frac{d}{d \lambda} \right)^m(P_\lambda
f)^\delta(\textbf{D}; x; \textbf{E})\bigg\|_{\textbf{2}}(1+ |x|)^n (1+ |\lambda|)^s
\varphi_0^{-r_p}(x)
\nonumber\\
&\leq \sum_{\ell=0}^m c_\ell (1+ |x|)^{n+u} (1+ |\lambda|)^{s+q}
(1+|\delta|)^q \varphi_0^{1-r_p}(x)e^{\e |x|}
\bigg\|\left(\frac{d}{d \lambda} \right)^{m-\ell}
\widetilde{f^\delta}(\lambda)\bigg\|_{\textbf{2}}. \nonumber\\
\intertext{ We notice that for all $0 < p \leq 2$, $1-r_p >0$. Now we make use of the estimate
$\varphi_0^{1 -r_p}(x) \leq
a(1+|x|)^{b_p}~ e^{(r_p-1)|x|},~~  x \in X$ (\emph{where  $b_p$ is a positive real number, to be precise, it is exactly $1 -r_p$} ). This is an easy consequence of the  two-sided estimate (\ref{eq:pri:apraori estimate of phi 0}) of the elementary spherical function $\varphi_0(x)$. Hence we can continue the above chain of inequalities by}
& \leq \sum_{\ell=0}^m c_\ell (1+ |x|)^{n+u+b_p} (1+ |\lambda|)^{s+q}
(1+|\delta|)^{q} e^{-\gamma |x|} \bigg\|\left(\frac{d}{d \lambda}
\right)^{m-\ell} \widetilde{f^\delta}(\lambda)\bigg\|_{\textbf{2}},
\end{align}
where  $\gamma = 2-\frac{2}{p}-r_p$. Clearly, $\gamma > 0$ as $ r_p <
\frac{2p-2}{p}$. Hence from (\ref{No. 2})
\begin{align}\label{NO-3}
\sup_{x\in G, \lambda \in Int\mathfrak a^*_\e, }&
\bigg\|\left(\frac{d}{d \lambda} \right)^m(P_\lambda f)^\delta(\textbf{D}; x;
\textbf{E})\bigg\|_{\textbf{2}}(1+ |x|)^{n} (1+
|\lambda|)^{s}  \varphi_0^{-r_p}(x) \nonumber\\
& \leq \sum_{\ell=0}^t c_\ell (1+|\delta|)^{q}\left\{\sup_{x \in X} (1+
|x|)^{n+u+b_p}
e^{-\gamma |x|}\right\} \nonumber\\
& \hspace{1.4in} \left\{\sup_{\lambda \in Int\mathfrak a^*_\e}(1+
|\lambda|)^{s+q}
\bigg\|\left(\frac{d}{d \lambda}
\right)^{m-\ell}
\widetilde{f^\delta}(\lambda)\bigg\|_{\textbf{2}}\right\} \nonumber \\
& \leq \sum_{\ell=0}^m \overline{c_{\ell}}~(1+|\delta|)^{q} \left\{\sup_{\lambda \in
Int\mathfrak a^*_\e}(1+ |\lambda|)^{s+q} \bigg\|\left(\frac{d}{d
\lambda}
\right)^{m-\ell} \widetilde{f^\delta}(\lambda) \bigg\|_{\textbf{2}}\right\}.
\intertext{ Now the expression within braces is the norm $\tau_{s+q, m-\ell}(\widetilde{f^\delta})$. Using
the continuity of the $\delta$-spherical transform, we write:~ there exists
 positive integers
$l,~ t$ such that the last expression (\ref{NO-3}) is dominated by   }
&c  (1+|\delta|)^{q} \sup_{x \in G} \|\textbf{L}^l
f^\delta(x) \|_{\textbf{2}} (1+|x|)^{t} \varphi_0^{-\frac{2}{p}}(x).
\end{align}
\end{proof}
\begin{Remark}
\label{rem:ch2:zeros of P lambda f x}
The fact that  $\lambda \mapsto Q_\delta(1 -i \lambda)^{-1} (P_\lambda f)^\delta$ (for all $f \in \mc S^p(X)$) is holomorphic on $Int \mf a^*_\e$  can be given a separate proof by using the structural form (\ref{structural form}) of $(P_\lambda f)^\delta(\cdot)$. It can be shown that (in fact we shall discuss about this in detail in the next section), for each $x = k a_t K$, $\Phi_{\lambda, \delta}(k a_t K)= \delta(k) Q_\delta(1 + i\lambda) \Phi(\lambda, t)$, where $\Phi(\lambda, t)$ is a scalar valued function on $\C \times \overline{\R^+}$ such that for each value of $\lambda$ it is a nonzero function in the $t$ variable. Hence, by (\ref{structural form}),  $Q_\delta(1 + i \lambda)^{-1}~(P_\lambda f)^\delta$ is holomorphic on $Int \mf a^*_\e$. Now $\lambda \mapsto (P_\lambda f)^\delta$ being an even function, it is  easy to notice that actually, $\lt[Q_\delta(1 -i\lambda) Q_\delta(1+ i \lambda) \rt]^{-1} (P_\lambda f)^\delta$ is holomorphic on $Int \mf a^*_\e$.
\end{Remark}
The above proposition helps us to get the decay/growth  of $P_\lambda f$
when $f \in \mc S^p(X)$ and  $\lambda \in \mathfrak a^*_\e$. The following  is the main theorem of this section.
\begin{Theorem}
\label{ theo: main theorem spectral proj}
For $f \in \mc S^p(X)$ ($0 < p \leq 2$),  the complex valued function  $P_\lambda f$ defined on $\mf a^*_\e \times X$  has the properties:
\begin{enumerate}
\item For each $\lambda \in \mathfrak a^*_\e$, ~~$P_\lambda f \in \mathcal E_\lambda(X)$;
\item For each $x \in X$ the function $\lambda \mapsto P_\lambda
f(x)$ is an even holomorphic function on $Int \mf a^*_\e$ and it extends as an even continuous function to the closed strip $\mf a^*_\e$. For each $\delta \in \what{K}_M$, the $\delta$-projection $(P_\lambda f)^\delta$ is an identically zero function on $X$ at all the zeros of the Kostant polynomial $Q_\delta(1 - i\lambda)$ lying in $Int \mf a^*_\e$;
\item For each $\textbf{D}, \textbf{E} \in \mathcal U(\mathfrak g)$, $m, n, s \in \Z^+ \cup
\{0\}$ and for all real number  $ r_p < \frac{2p-2}{p}$, one can find integers $\ell, t \in \Z^+$ and a positive constant $c$ depending on $m, n, s$ and $r_p$ such that:
\begin{align}
&\sup_{x \in G, \lambda \in Int\mathfrak a^*_\e }
\bigg|\left(\frac{d}{d \lambda} \right)^m P_\lambda f(\textbf{D}; x; \textbf{E})
\bigg| (1+ |x|)^n (1+ |\lambda|)^s \varphi_0^{-r_p}(x) \nonumber\\
& \hspace{1.6in}\leq c \sup_{x \in G} |\textbf{L}^\ell f(x)| (1 + |x|)^t \varphi_0^{-\f{2}{p}}(x) < +\infty.
\end{align}
\end{enumerate}
\end{Theorem}
\begin{proof}
The property (i) has already been discussed. Condition
(ii) is an easy consequence of the  Peter-Weyl decomposition
\be
\label{eq:Peter-Weyl decomposition of P lambda}
P_\lambda f(x) = \sum_{\delta \in \what{K}_M} tr (P_\lambda f)^\delta(x),\hspace{.2in} \mbox{for all~} x \in X
\ee
 and Proposition \ref{proposition-for decay of delta proj of Pf}. To obtain
(iii) we use (\ref{eq:Peter-Weyl decomposition of P lambda}) to get:
\begin{align}\label{No.4}
\bigg|\left(\frac{d}{d \lambda} \right)^m P_\lambda f(\textbf{D}; x; &\textbf{E})
\bigg| (1+ |x|)^n (1+ |\lambda|)^s \varphi_0^{-r_p}(x) \nonumber\\
& \leq~ \sum_{\delta \in \widehat{K}_M} \bigg|\left(\frac{d}{d
\lambda} \right)^m tr(P_\lambda f)^\delta(\textbf{D}; x; \textbf{E}) \bigg| (1+ |x|)^n
(1+ |\lambda|)^s \varphi_0^{-r_p}(x) \nonumber\\
&=~ \sum_{\delta \in \widehat{K}_M} \bigg| tr\bigg[ \hspace{-.03in}\left(\frac{d}{d
\lambda} \right)^m\hspace{-.06in}(P_\lambda f)^\delta\hspace{-.03in}\bigg](\textbf{D}; x; \textbf{E}) \bigg| (1+
|x|)^n (1+ |\lambda|)^s \varphi_0^{-r_p}(x) \nonumber\\
& \leq~ \sum_{\delta \in \widehat{K}_M}\bigg\|\left(\frac{d}{d
\lambda} \right)^m\hspace{-.05in}(P_\lambda f)^\delta(\textbf{D}; x; \textbf{E}) \bigg\|_{\textbf{2}}(1+
|x|)^n (1+ |\lambda|)^s \varphi_0^{-r_p}(x).\\
&\mbox{The next  inequality follows easily from (\ref{No.4}) by applying (\ref{NO-3}).}\nonumber
\end{align}
\begin{align}\label{No.5}
\sup_{x \in G; \lambda \in Int \mathfrak a^*_\e}&
\bigg|\left(\frac{d}{d \lambda} \right)^m P_\lambda f(\textbf{D}; x; \textbf{E})
\bigg| (1+ |x|)^n (1+ |\lambda|)^s \varphi_0^{-r_p}(x) \nonumber\\
&\leq~ c~\sum_{\delta \in \what{K}_M} \sum_{j=0}^m \sup_{\lambda \in Int\mathfrak a^*_\e} \lt\{ (1+ |\lambda|)^{s+q} (1+|\delta|)^{q}
\bigg\|\left(\frac{d}{d \lambda} \right)^{m-j}
\widetilde{f^\delta}(\lambda)\bigg\|_{\textbf{2}}\rt\} \nonumber\\
&\leq \sum_{j=1}^m \sum_{\delta \in \what{K}_M} \hspace{-.05in}(1 + |\delta|)^{-2} \lt\{ \sup_{\lambda \in Int\mathfrak a^*_\e}  (1+ |\lambda|)^{s+q} (1+|\delta|)^{q+2}
\bigg\|\left(\frac{d}{d \lambda} \right)^{m-j}
\widetilde{f^\delta}(\lambda)\bigg\|_{\textbf{2}}\rt\}.
\end{align}
As $f \in \mc S^p(X)$, so its HFT  $\mc F f \in \mc S(\mf a^*_\e \times K/M)$ \cite{Eguchi76}.
  We have also noticed that the  Fr\'{e}chet topology on $\mc S(\mf a^*_\e \times K/M)$ is induced by the countable family \ref{SW semi norm for image of HFT} of seminorms.
     By the theory of smooth functions on the compact group \cite[Theorem 4]{Sugiura71}, it follows that the topology of $\mc S(\mf a^*_\e \times K/M)$ can also be obtained from the equivalent family of seminorms, given by
    $$\sup_{\lambda \in Int \mf a^*_\e; \delta \in \what{K}_M} \lt\|P\lt(\f{d}{d \lambda}\rt) \widetilde{f^\delta}(\lambda)\rt\|_{\textbf{2}} (1 + |\lambda|)^{\texttt{n}} (1+ |\delta|)^{\texttt{m}} < + \infty.$$
     Hence, we can state that the expression within braces of  each of the summands of (\ref{No.5})   is dominated by the single finite quantity :
\begin{equation}
\label{eq:no-1}
\sup_{\lambda \in Int\mf a ^*_\e; \delta \in \what{K}_M} \lt\|\lt(\f{d}{d \lambda} \rt)^{m-j} \widetilde{f^\delta}(\lambda)  \rt\|_{\textbf{2}} (1 + |\lambda|)^{s +q} (1 + |\delta|)^{q+2}.
\end{equation}
This coupled with the summability of $\sum (1+ |\delta|)^{-2}$ reduces the inequality (\ref{No.5})  to:
\begin{align}
\label{ali:No-6}
&\sup_{x \in X, \lambda \in Int \mathfrak a^*_\e}
\bigg|\left(\frac{d}{d \lambda} \right)^m P_\lambda f(\textbf{D}; x; \textbf{E})
\bigg| (1+ |x|)^n (1+ |\lambda|)^s \varphi_0^{-r_p}(x) \nonumber\\
& \hspace{.5in}\leq c \sum_{j=0}^m \lt\{ \sup_{\lambda \in Int\mf a ^*_\e; \delta \in \what{K}_M} \lt\|\lt(\f{d}{d \lambda} \rt)^{m-j} \widetilde{f^\delta}(\lambda)  \rt\|_{\textbf{2}} (1 + |\lambda|)^{s +q} (1 + |\delta|)^{q+2}\rt\}. \nonumber
\intertext{Again by the equivalence of the seminorms on $\mc S(\mf a^*_\e \times K/M)$, we can find positive integers $m_1, m_2, m_3$ such that the last expression is}
&\hspace{.5in}\leq c\sup_{\lambda \in Int \mf a^*_\e; k \in K/M} \lt|\lt(\f{d}{d \lambda}\rt)^{m_1} \mc F f(\lambda, k ; \omega_{\mf k}^{m_2}) \rt| (1 + |\lambda|)^{m_3}, \nonumber
\intertext{and  by the continuity of the HFT on the Schwartz space $\mc S^p(X)$\cite{Eguchi76}, we get nonnegative integers $\ell, t \in \Z^+$ such that the above expression is}
&\hspace{.5in} \leq c \sup_{x \in G} |\textbf{L}^\ell f(x)| (1 +|x|)^{t} \varphi_0^{- \f{2}{p}}(x).
\end{align}
\vspace{-.05in}
This completes the proof of theorem.
\end{proof}

\begin{Remark}
This part of the characterization does not really use the fact that $G$ is of real rank one. We state the analogue of Theorem \ref{ theo: main theorem spectral proj}  for any Riemannian symmetric space  $X=G/K$ where $G$ a non-compact, connected semisimple Lie group with finite center and not necessarily of real rank one.
(We continue to use the notation we have established for the real rank one case, the modifications needed are mostly obvious.)\\

Let $f \in \mc S^p(X)$  with $0 < p \leq 2$, then $P_\lambda f$ is a complex valued function on $\mf a^*_\e \times X$ where $\mf a^*_\e\hspace{-.07in} = \hspace{-.07in}\lt\{\lambda \in \mf a^*_\C \simeq \C^n|~ \|w \lambda (X)\| \hspace{-.05in}\leq \e \rho(X) \mbox{~for all~} X \in \mf a \mbox{~and~} w \in W \rt\}$ and $P_\lambda f$ satisfies the following properties:
\begin{enumerate}
  \item For each $\lambda \in \mf a^*_\e$, $P_\lambda f \hspace{-.05in}\in \mathcal E_\lambda(X)\hspace{-.05in}= \hspace{-.05in} \lt\{g \in \mathcal C^\infty(X)|\mathbf{L}~  g = - (\|\rho\|^2+ \langle \lambda, \lambda \rangle_1)g \rt\}$.
  \item For each $x \in X$, $\lambda \mapsto P_\lambda f(x)$ is a $W$-invariant holomorphic function on the interior of the complex tube $\mf a^*_\e$ and it extends as a continuous function to the closed tube.\\
      For  each $\delta \in \what{K}_M$, the $\delta$-projection $(P_\lambda f)^\delta$ is a $d_\delta \times d_\delta$ matrix valued function, where $d_\delta = dim V_\delta$. Also, $x \mapsto (P_\lambda f)^\delta(x)$  identically vanishes on each hypersurface in $Int \mf a^*_\e$ on which the polynomial $det Q_\delta(\rho - i \lambda)$ vanishes. (Here  $Q_\delta(\rho- i \lambda)$ is a $\ell_\delta \times \ell_\delta$ ($\ell_\delta = dim V_\delta^M$) matrix of constant coefficient polynomials in $(\rho-i \lambda)$, being the higher rank analogue of the Kostant polynomial \cite[Ch~III, \S 3]{Helgason-gas}.)
  \item For each $\textbf{D}, \textbf{E} \in \mathcal U(\mathfrak g)$, $Q \in \mathbf{S}(\mf a)$,  $ n, s \in \Z^+ \cup
\{0\}$ and for all real number  $ r_p < \frac{2p-2}{p}$, one can find integers $\ell, t \in \Z^+$ and a positive constant $c$ depending on $n, s, r_p$ and degree of the polynomial $Q$ such that:
\begin{align}
&\sup_{x \in G, \lambda \in Int\mathfrak a^*_\e }
\bigg|Q(\partial \lambda) P_\lambda f(\textbf{D}; x; \textbf{E})
\bigg| (1+ |x|)^n (1+ \|\lambda\|)^s \varphi_0^{-r_p}(x) \nonumber\\
& \hspace{1.6in}\leq c \sup_{x \in G} |\textbf{L}^\ell f(x)| (1 + |x|)^t \varphi_0^{-\f{2}{p}}(x) < +\infty\nonumber,
\end{align}
where $Q(\partial \lambda) := Q\lt(\f{\partial}{\partial \lambda_1}, \cdots, \f{\partial}{\partial \lambda_n} \rt)$.
\end{enumerate}
\end{Remark}

For each $\e >0$, let us now define a function space $\mc P_{\e}(X)$.
\begin{Definition}
\label{def:image space under spectral projn}
For $\e > 0$, then $\mc P_\e(X)$ denotes the class of functions $(\lambda, x) \mapsto f_\lambda(x)$ defined on $\mf a^*_\e \times X$ and satisfying the following conditions:
\begin{enumerate}
  \item For each $x \in X$ the function $\lambda \mapsto f_\lambda(x)$ is an
even $\mc C^\infty$ function on $\mathfrak a^*$ and is analytic
on the interior of the strip $\mathfrak a^*_\e =\{ \lambda~ |~ |\Im{\lambda}| \leq \e \}$. On the
boundary it extends as a continuous function.
  \item For each $\lambda \in \mathfrak a^*_\e$
the map $x \mapsto f_\lambda(x)$ is a $\mc C^\infty$ function on $X$, an eigenfunction of $\textbf{L}$, $f_\lambda \in \mathcal E_\lambda(X) $. \\
Moreover,  for each $\delta \in \what{K}_M$ and $x \in G$, the function $\lambda \mapsto   Q_\delta(1 -i \lambda)^{-1}f_{\lambda}^\delta(x) $ is holomorphic on $Int\mf a^*_\e$.
  \item  For each $\textbf{D},\textbf{E} \in \mathcal U
(\mathfrak g_\C) $ and $m, n, s \in \Z^+ \cup\{0\}$
\begin{equation}
\label{eq:decay/growth }
\sup_{x \in G, \lambda \in Int\mathfrak a^*_\e} \bigg|
\left(\frac{d}{d \lambda}  \right)^m f_\lambda(\textbf{D}; x; \textbf{E}) \bigg|
(1+|x|)^n (1+|\lambda|)^s \varphi_0^{-r_\e} (x) < + \infty,
\end{equation}
where $r_\e < \frac{1 - \e}{1+ \e}$.
\end{enumerate}
\end{Definition}
It is easy to verify that $\mc P_{\e}(X)$ is a Fr\'{e}chet space with the topology induced by the countable family of seminorms (\ref{eq:decay/growth }).\\
 We shall conclude this section by restating the Theorem \ref{ theo: main theorem spectral proj} in the light of Definition \ref{def:image space under spectral projn}.
 \begin{Theorem}
 \label{theo: main theorem spectral proj-2nd form}
 The spectral projection $f \mapsto P_\lambda f$ is a continuous map from the Schwartz space $\mc S^p(X)$ ($0 < p \leq 2$)  into $\mc P_{\e}(X)$, where $\e = \lt( \f{2}{p}-1 \rt)$.
 \end{Theorem}
In the next section we shall obtain  sufficient conditions for the image of $\mc S^p(X)$ under
the transform $f \mapsto P_\lambda f$. The fact that $G$ is of real rank-one  plays a crucial role
there.


\section{\textbf{Sufficient Conditions}}
\label{sec: Sufficient Conditions}

We begin this section with the definition of a specific subspace of the function space $\mc P_{\e}(X)$ for each $\e \geq 0$.
\begin{Definition}
\label{def:K-finite subspace of teh image under sp-projc}
We denote by $\mc P_{\e}(X)_K$, for each $\e > 0$, the class of functions $f_\lambda(x)$ in $\mc P_\e(X)$ which are of left-$K$-finite type in the $x$ variable, where the finite set of $\delta \in \what{K}_M$ involved can be chosen independently of $\lambda$.
\end{Definition}
In this section we shall try to establish a sufficient condition for
a measurable function $(\lambda, x)\mapsto f_{\lambda}(x) \in \mc P_\e(X)_K$ to be of the form $P_\lambda f(x)$ with some
$f \in \mc S^p(X)$ for  suitable $0 < p\leq 2$. Let us fix one $\e \geq 0 $ and
 a function $f_\lambda(x) \in \mc P_\e(X)_K$.

Because of the decay (\ref{eq:decay/growth })
the  integral
\begin{equation}
f_n(x) := (-1)^n\int_{\mathfrak a^{*+}}~ (1+ \lambda^2)^n
f_\lambda(x) ~ |\textbf{c}(\lambda)|^{-2} ~d \lambda, ~~ (~n \in \Z^+~)
\end{equation}
converges absolutely,
where $\textbf{c}(\lambda)$ is the Harish-Chandra $\textbf{c}$-function.
Let us set $f_0=f $, i.e
\begin{equation}
\label{Inversion}
f(x) = \int_{\mathfrak a^{*+}}
f_\lambda(x) |\textbf{c}(\lambda)|^{-2} d \lambda.
\end{equation}
It also follows from the specified decay (\ref{eq:decay/growth }) of the function $(\lambda, x) \mapsto f_\lambda(x)$:
 that for $n = 0, 1, \cdots, $
~  the function $f_n \in \mc C^\infty(X)$. As
$f_\lambda(\cdot) \in \mathcal E_\lambda(X)$, it can be shown that
$\textbf{L}^n f = f_n$ .\\
For each $\lambda \in \mathfrak a^*_\e$ and  $\delta \in \widehat{K}_M$ we define the operator valued left $\delta$-projection by
\begin{equation}
\label{eq:delta projection of f-lambda-x}
f_\lambda^\delta(x) = \int_K f_\lambda(kx) \delta(k^{-1}) dk.
\end{equation}
 It is clear from the definition of the function space $\mc P_\e(X)_K$ that for each $\delta \in \widehat{K}_M$ and
each $\lambda \in \mathfrak a^*_\e$,
$f_\lambda^\delta \in \mathcal E_\lambda^\delta(X)$.
The  Peter-Weyl decomposition of the function $f_\lambda(\cdot)$ is as follows
\begin{equation} \label{breaking of main function}
f_\lambda(x) = \sum_{\delta \in \what{K}_M}trf_\lambda^\delta(x), ~~~~~\lambda  \in \mf a^*_\e.
\end{equation}
As $f_\lambda$ is assumed to be left $K$-finite, so in the above decomposition (\ref{breaking of main function}) all but finitely many terms are identically zero functions. Let us denote $\Gamma(f)$ for the finite subset of $\what{K}_M$ corresponding to the function $f_\lambda$ for which the summands are non zero functions. It follows from the earlier discussion that for each $\delta \in \Gamma(f)$ and $\lambda \in \mf a^*_\e$, $x \mapsto tr f_\lambda^\delta(x)$ is of left $\check{\delta}$-type. Hence $trf_\lambda^\delta(\cdot) \in \mc E_\lambda(\check{\delta}, X)$.
\begin{Lemma}
For  each $\delta \in \Gamma(f)$ the map $(\lambda, x) \mapsto f_\lambda^\delta(x)$
satisfies the decay
\begin{equation}\label{decay of operator valued projection}
\sup_{x \in X; \lambda \in Int\mathfrak a^*_\e } \bigg\|
\bigg(\frac{d}{d \lambda}  \bigg)^m f^\delta_\lambda(x) \bigg\|_{\mathbf{2}} (1+|x|)^n (1+|\lambda|)^s
\varphi_0^{-r_\e}(x) < \texttt{K} <+\infty
\end{equation}
where  $\texttt{K} = d_\delta^{3/2} \cdot c$, the constant $c$ being independent of $\delta$.
\end{Lemma}
\begin{proof}
The assertion is true because
\begin{align*}
&\lt\{ \lt( \frac{d}{d \lambda} \rt)^m f^\delta_\lambda(x) \rt\}  (1+|x|)^n
(1+|\lambda|)^s \varphi_0^{-r_\e}(x)\\
&\hspace{.4in} = d_\delta \left\{\int_K \bigg(\frac{d}{d \lambda}  \bigg)^m f_\lambda(kx)
\delta(k^{-1}) dk \right\} (1+|x|)^n (1+|\lambda|)^s \varphi_0^{-r_\e}(x) \\
&\hspace{.4in} =  d_\delta \int_K \left\{\bigg(\frac{d}{d \lambda}  \bigg)^m f_\lambda(kx)\right\}
\delta(k^{-1})  (1+|kx|)^n (1+|\lambda|)^s \varphi_0^{-r_\e}(kx)  dk.
\end{align*}
Now taking Hilbert Schmidt norms on both sides and using the fact $\|\delta(k)\|_{\textbf{2}}
= \sqrt{d_\delta} $ we get an inequality from which we get
the required conclusion. (\emph{ It is easy to check that the $\delta$ dependent part in the dominating constant is precisely $d_\delta^{\frac{3}{2}}$}).
\end{proof}
 An immediate corollary is the following.
\begin{Corollary}
For each $m, n, s \in \Z^+ \cup \{0\}$
and for each fixed $\delta \in \Gamma(f)$
\be\label{decay of the trace}
\sup_{x \in X; \lambda \in Int\mathfrak a^*_\e } \bigg|
\bigg(\frac{d}{d \lambda}  \bigg)^mtrf^\delta_\lambda(x) \bigg | (1+|x|)^n (1+|\lambda|)^s
\varphi_0^{-r_\e}(x)  <+\infty.
\ee
\end{Corollary}

\begin{Lemma}
The function $f$
obtained in (\ref{Inversion}) is also left-$K$-finite and moreover,
\begin{equation}
trf^\delta(x) = \int_{\mathfrak a^{*+}} trf_\lambda^\delta(x)~
|\textbf{c}(\lambda)|^{-2} d \lambda ~.
\end{equation}
\end{Lemma}
\begin{proof}
Using (\ref{Inversion}) and (\ref{breaking of main function})
we get the following:
\begin{align}\label{1-lemma}
f(x) &= \int_{\mathfrak a^{*+}} \sum_{\delta \in \Gamma(f)}trf_\lambda^\delta(x)
~|\textbf{c}(\lambda)|^{-2} d \lambda \nonumber,\\
&= \sum_{\delta \in \Gamma(f)} \int_{\mathfrak a^{*+}}trf_\lambda^\delta(x)
~|\textbf{c}(\lambda)|^{-2} d \lambda.
\end{align}
Let us denote $\psi_\delta(x) = \int_{\mathfrak a^{*+}}trf_\lambda^\delta(x)
~|\textbf{c}(\lambda)|^{-2} d \lambda$. The integral converges absolutely because of the decay
(\ref{decay of the trace}).  We have already noticed that $tr f_\lambda^\delta(\cdot)$ are of left $\check{\delta}$-type and we only need the routine checking
\begin{align}
\psi_\delta(x) &= d_\delta \int_{\mathfrak a^{*+}} \{ \chi_{\check{\delta}}
\ast trf_\lambda^\delta~ \}(x) ~|\textbf{c}(\lambda)|^{-2} d \lambda\nonumber\\
&=d_\delta \int_{\mathfrak a^{*+}} \left\{ \int_K \chi_{\check{\delta}}
(k^{-1} )tr f_\lambda^\delta (kx) dk  \right\} ~|\textbf{c}(\lambda)|^{-2}
d \lambda \nonumber\\
&=d_\delta \int_{K} \chi_{\check{\delta}} (k^{-1} ) \left
\{ \int_{\mathfrak a^{*+}}trf_\lambda^\delta (kx) dk   ~|\textbf{c}(\lambda)|^{-2} d\lambda \right\}
dk  \nonumber\\
&= d_\delta \{ \chi_{\check{\delta}} \ast \psi_\delta\}(x)
\end{align}
to conclude that  each $\psi_\delta$ is a scalar valued left $\check{\delta}$-type.
The rest follows from the Peter-Weyl decomposition.
\end{proof}
So far we have noted that the function $f$ obtained in (\ref{Inversion})
is in $\mc C^\infty(X) $ and it is of left-$K$-finite type. Now we shall try to show that
$f \in \mc S^p(X)$ for some $0 < p \leq 2$. Towards that we shall first try to obtain a structural form of $f_\lambda^\delta$
analogous to the one given in Proposition \ref{delta projection of P-lambda}.
The assumption that $f_\lambda^\delta \in \mc E_\lambda^\delta(X)$ will now play a crucial role. The following theorem is the key to the desired  form of $f_\lambda^\delta$. Let $\delta \in \what{K}_M$ and $V_\delta$ ($d_\delta= dim V_\delta$) be the representation space for $\delta$ with the orthonormal basis $ v_1, v_2, \cdots, v_{d_\delta}$ where $V_\delta^M = \C v_1$.
\begin{Theorem}
\label{theo:basis of the eigen space}
\emph{[Helgason, \cite[Theorem 1.4, p-133]{Helgason70}]}\\
Let  $\lambda \in \C$ be such that  $\Re \langle i \lambda, \alpha \rangle \geq 0$, where $\alpha$ is positive restricted  root. Then the functions
\begin{equation}
\label{eq:basis vector of the eigen space}
\Psi_{\lambda,\check{\delta}j}(x)= \sqrt{d_\delta} \int_K e^{-(i \lambda+1) H(x^{-1}k)}
\left\langle \delta(k) v_1~,~v_j\right\rangle dk, \hspace{.5in} 1 \leq j \leq d_\delta,
\end{equation}
form a basis of the eigenspace $\mc E_\lambda(\check{\delta}, X)$.
\end{Theorem}
We note that by using the definition  (\ref{eq:pri:gen. sp. funct.}) of the generalized spherical functions we can write the basis vectors as follows
\begin{equation}
\label{eq:basis vector of the eigen space,alt form}
\Psi_{\lambda,\check{\delta}j}(x)= \sqrt{d_\delta} \left\langle \Phi_{\lambda, \delta}(x)
v_1~,~v_j\right\rangle.
\end{equation}
\begin{Remark}
\label{rem:reark about the simple lambdas}
For a  group $G$ with real rank one, we have identified the Iwasawa-$A$-subgroup with $\R$. With this normalization $\mf a$, $\mf a^*$ are identified with $\R$ and $\mf a^+$, $\mf a^{*+}$  with $\R^+$. As we are only considering the real rank one group, so there will be a smallest positive restricted root $\alpha$ and at most one more which will be $2 \alpha$. Clearly, $\alpha \in \R^+$. This immediately suggests that for all $\lambda $ with $\Im \lambda \leq 0$, $\Re \langle i \lambda, \alpha \rangle \geq 0 $.\\
Hence, for all $\lambda \in  \{ \lambda \in \mf a^*_\e ~|~ \Im \lambda \leq 0\}$, the vectors $\Psi_{\lambda,\check{\delta}j}(x)$ forms a basis of $\mc E_\lambda(\check{\delta}, X)$.
\end{Remark}

\begin{Lemma}\label{structure Lemma}
For each $\delta \in \widehat{K}_M$ the matrix valued projection
$f^\delta_\lambda $ of $f_\lambda$ for each $\lambda \in \mathfrak a^{*}_\e$
has the following structural form
\be\label{structural form-2}
f^\delta_\lambda(k a_t.0) = \sqrt{d_\delta}~\Phi_{\lambda, \delta}
(ka_t.0) h^\delta(\lambda)
\ee
where  $\Phi_{\lambda, \delta}(k a_t. 0)$ is a $(d_\delta \times 1)$
matrix and $h^\delta(\lambda)$ is a $(1 \times d_\delta) $ matrix
\end{Lemma}
\bpf
We note that by assumption the function $\lambda \mapsto f^\delta_\lambda(\lambda)$ is even for all $x \in X$. Therefore it is enough to establish the structural form (\ref{structural form-2}) for $\lambda$'s in $\mathfrak a^{*}_\e$ with $\Im \lambda \leq 0$.\\
We have noticed that $tr f_\lambda^\delta \in \mathcal E_{\lambda}( \check{\delta}, X)$ for all $\lambda \in
\mathfrak a^*_\e$ and $\delta \in \Gamma(f)$.
Hence for $\lambda \in  \{ \lambda \in \mf a^*_\e ~|~ \Im \lambda \leq 0\}$ we write $tr f^\delta_\lambda (x)$ in terms of the basis vectors given in
(\ref{eq:basis vector of the eigen space,alt form}) as follows.
\begin{align}\label{I}
trf^\delta_\lambda(x) = \sqrt{d_\delta}~
\sum_{j=1}^{d_\delta} h^\delta_j(\lambda) \left\langle    \Phi_{\lambda, \delta}(x) v_1,  v_j
\right\rangle,
\end{align}
where $h^\delta_j(\lambda)$ are coefficients depending on
$\lambda$. Let us denote the $(1 \times d_\delta)$ matrix $h^\delta(\lambda)
= \left( h_1^\delta(
\lambda), \cdots, h_{d_\delta}^\delta(\lambda) \right)$. We also recall the fact that the generalized spherical function $\Phi_{\lambda, \delta}(x)$ vanishes on the orthogonal complement of $V_\delta^M$, so we can regard $\Phi_{\lambda, \delta}(x)$ as the $(d_\delta \times 1)  $ column vector with the entries $\langle \Phi_{\lambda, \delta}(x)v_1, v_j \rangle $. Then it is clear from
(\ref{I}) that
\be
\label{II}
trf^\delta_\lambda(x) = \sqrt{d_\delta}~tr
[ \Phi_{\lambda, \delta}(x) h^\delta(\lambda)].
\ee
Next we shall show that the matrices
$f^\delta_\lambda(x) $ and $\Phi_{\lambda, \delta}(x) h^\delta(\lambda)$ have identical entries.
\begin{align} \label{III}
f_\lambda^\delta(x)_{\imath~\ell} & = \left\langle  f_\lambda^\delta(x) v_\ell, v_\imath \right\rangle \nonumber\\
&= d_\delta\left\langle \int_K trf^\delta_\lambda(kx) \delta(k^{-1}) dk v_\ell, v_\imath   \right\rangle
\nonumber\\
&= d_\delta \int_Ktrf^\delta_\lambda(kx) \left\langle \delta(k^{-1})  v_\ell, v_\imath \right\rangle dk.
\end{align}
Now we use (\ref{II}) to replace $tr f^\delta_\lambda(kx)$ to get
\begin{align}\label{IV}
f_\lambda^\delta(x)_{\imath~\ell} & = d_\delta^{\frac{3}{2}} \sum_{j =1}^{d_\delta}
\int_K  \left\langle \Phi_{\lambda, \delta}(kx) v_1~,~v_j\right\rangle
~\left\langle \delta(k^{-1})  v_\ell, v_\imath \right\rangle  h^\delta_j(\lambda)~dk \nonumber\\
&= d_\delta^{\frac{3}{2}} \sum_{j =1}^{d_\delta} \int_K
\left\langle \delta(k) \Phi_{\lambda, \delta}(x) v_1, v_j\right\rangle ~\left\langle
\delta(k^{-1})  v_\ell, v_\imath \right\rangle  h^\delta_j(\lambda)~dk \nonumber\\
&=d_\delta^{\frac{3}{2}} \sum_{j =1}^{d_\delta} \int_K
\overline{\left\langle \delta(k^{-1}) v_j,  \Phi_{\lambda, \delta}(x) v_1\right\rangle}
~\left\langle \delta(k^{-1}) v_\ell, v_\imath \right\rangle  h^\delta_j(\lambda)~dk.
\end{align}
The representation coefficients $k \mapsto \left\langle
\delta(k)v, u\right\rangle$ ($u, v \in V_\delta$) satisfy the following consequences of the  `Schur's Orthogonality
Relations' : If $u, v, u^\prime, v^\prime \in V_\delta$, then
\be
\label{Schur orthogonality}
\int_K \left\langle \delta(k) u, v  \right\rangle
\overline{\left\langle \delta(k)u^\prime,  v^\prime  \right\rangle} =
d_\delta^{-1} \left\langle u, u^\prime  \right\rangle \overline{\left\langle v, v^\prime
\right\rangle}.
\ee
Using (\ref{Schur orthogonality}) in (\ref{IV}) as also the fact that $\{v_i\}(1 \leq i \leq d_\delta)$ forms an
orthonormal basis of the representation space $V_\delta$ we write
\be
\label{V}
f_\lambda^\delta(x)_{\imath~\ell} = \sqrt{d_\delta} \left\langle
\Phi_{\lambda, \delta}(x)v_1, v_\imath   \right\rangle h^\delta_\ell(\lambda).
\ee
The right hand side of (\ref{V}) is precisely the $(\imath, \ell)$ entry
of the matrix $\sqrt{d_\delta}~ \Phi_{\lambda, \delta}(x) h^\delta(\lambda)$.
Hence the Lemma follows.
\epf

\begin{Remark}
\label{properties of h-delta}
\begin{align}
\label{ali:evenness of h delta}
\mbox{Write\hspace{.3in}}f_\lambda^\delta(x) &= \sqrt{d_\delta}  \Phi_{\lambda, \delta}(x) h^\delta(\lambda)\nonumber\\
&= \sqrt{d_\delta} \{ Q_\delta(1 - i \lambda) \Phi_{\lambda, \delta}(x)\} \{Q_\delta(1 -i \lambda)^{-1} h^\delta(\lambda) \}.
\end{align}
 We notice that $f_\lambda^\delta$ is even in the $\lambda $ variable and  the function $\lambda \mapsto Q_\delta(1 - i \lambda) \Phi_{\lambda, \delta}(x)$ is even by (\ref{Intro- of kostant poly1}). Hence for all $\lambda \in \mf a^*_\e$, the function $\lambda \mapsto Q_\delta(1 -i \lambda)^{-1} h^\delta(\lambda)$ is an even function.
\end{Remark}
At this point we need to look in a different direction.
The matrix entries of the generalized spherical functions are associated with Jacobi
functions. Let $x = k a_t.0 \in X$. Then $$\Phi_{\lambda, \delta~j}(k a_t) = \langle
\Phi_{\lambda, \delta}(k a_t) v_1, v_j \rangle =\langle \delta(k)
\Phi_{\lambda, \delta}( a_t) v_1, v_j \rangle .$$ It can easily be seen that
$\Phi_{\lambda, \delta}(a_t) v \in V_\delta^M$ for all $v \in V_\delta$. Hence on $V_\delta^M$, $\Phi_{\lambda, \delta}(a_t)$ will be a multiplication operator
\begin{equation}
\label{eq:Phi as a multiplication operator}
\Phi_{\lambda, \delta}(a_t) v_1 = \varphi_{\lambda, \delta}(t) v_1,
\end{equation}
 where  $\varphi_{\lambda, \delta}(t)$ is a
function of $t$ depending on $\lambda$ and $\delta$. For each $\delta \in \widehat{K}_M$ and
$\lambda \in \mathfrak a^*_\C$
    the
function $ t \mapsto \varphi_{\lambda , \delta}(t)$ has an expression  in terms of the hypergeometric functions (Helgason \cite{Helgason74}, Koornwinder \cite{Koorn84} )
\begin{equation}
\label{Jacobi function}
\varphi_{\lambda, \delta}(t) = Q_\delta(i \lambda+1) (\alpha +1)^{-1}_r
(\sinh t)^r (\cosh t)^s \varphi_\lambda^{\alpha+r, \beta+s} (t),
\end{equation}
where  $\varphi_\lambda^{\alpha+r, \beta+s}$ is the Jacobi function
of the  first kind with parameters $(\alpha+r, \beta+s)$. The integers $(r,s)$ and the quantities
$\alpha , \beta$ are already introduced in (\ref{Structure of the Kostant's polynomial}). This Jacobi function has the integral representation \cite{Koorn84}:
\begin{equation}
\label{eq:integral representation of Jacobi funct}
\varphi_\lambda^{\alpha+r, \beta+s}(t)= \int_0^1 \int_0^{\pi} |\cosh{t}-\mf r e^{i \theta} \sinh{t} |^{- i \lambda -\varrho} dP_{\alpha+r, \beta+s}(\mf r, \theta),
\end{equation}
where  $\varrho = \alpha+r+\beta+s+1$ and $d P_{\alpha+r, \beta+s}(\mf r, \theta)$ is a probability measure \cite{Helgason87} on $[0,1]\times [0, \pi]$.
\begin{Lemma}
\label{lem:Jacobi at zero}
For all $\lambda \in \C$ the Jacobi function $\varphi_\lambda^{\alpha+r, \beta+s}$ satisfies the following:
\begin{align}
\label{Jacobi at zero-1}
\varphi_\lambda^{\alpha+r, \beta+s}(0)&= 1\\
\label{Jacobi at zero-2}
\lt|\lt(\f{d}{d \lambda} \rt)^k\varphi_\lambda^{\alpha+r, \beta+s}(t) \rt| &\leq c t^k e^{(|\Im{\lambda}|+\varrho)t}, ~~~t\in \R^+, ~k \in \Z^+.
\end{align}
 \end{Lemma}
 \begin{proof}
These two properties follow  from the integral representation (\ref{eq:integral representation of Jacobi funct}). We use the estimate
$$\log{|\cosh{t} - r e^{i \theta} \sinh{t}|} \leq c t$$ for all $t>0$
to get the inequality (\ref{Jacobi at zero-2}).
 \end{proof}

Our next Lemma concerns the domain on which the function $h^\delta$ is holomorphic.
\begin{Lemma}
\label{Lem:domain of holomorphicity of h-delta }
For each $\delta \in \what{K}_M$, the functions $\lambda \mapsto h^\delta(\lambda)$ and $\lambda \mapsto Q_\delta(1 -i \lambda)^{-1} h^\delta(\lambda) = g^\delta(\lambda)$  are holomorphic in the interior of the complex strip $\mf a^*_\e$.
 \end{Lemma}
 \begin{proof}
 We note that the zeros of the polynomial $Q_\delta(1 -i \lambda)$ are purely imaginary. We have assumed that $x \mapsto f_\lambda^\delta(x)$ is an identically zero function on $X$ for all $\lambda$ which are zeros of the polynomial $Q_\delta(1 -i \lambda)$. Also we have assumed that $f_\lambda^\delta$ is even in $\lambda$. So, $x \mapsto f_\lambda^\delta(x)$ is also zero for the zeros of $Q_\delta(1 + i \lambda)$ in  $Int \mf a^*_\e$. Hence, $\lambda \mapsto Q_\delta(1 +i\lambda )^{-1} f_\lambda^\delta(\cdot)$ is holomorphic on $Int \mf a^*_\e$.\\
 We restrict the function $f_\lambda^\delta(\cdot)$ to $(\mf a^*_\e \times A^+)$. Then by the structural form obtained in Lemma \ref{structure Lemma} we write:
 \begin{equation}
 \label{eq:yyy}
 f_\lambda^\delta(a_t) = \sqrt{d_\delta}~ \Phi_{\lambda, \delta}(a_t) h^\delta(\lambda),\hspace{.3in} (t>0).
 \end{equation}
 For proving  $\lambda \mapsto h^\delta(\lambda)$ is holomorphic on $Int \mf a^*_\e$ it is enough to prove that each of its matrix entries is so. By definition the $(1,j)$th ($1\leq j \leq d_\delta$) matrix entry of $f_\lambda^\delta(a_t)$ is given by $f_\lambda^\delta(a_t)_{1j}= \sqrt{d_\delta} \Phi_{\lambda, \delta~1}(a_t) h^\delta(\lambda)_j$.
 \begin{align}
 \label{ali:uuu}
 \Phi_{\lambda, \delta~1}(a_t) = \langle \Phi_{\lambda, \delta}(a_t) v_1, v_1 \rangle
 = \|v_1\| \varphi_{\lambda, \delta}(t).
 \end{align}
 Hence by (\ref{ali:uuu}) and the expression (\ref{Jacobi function}) we get:
 \begin{align}
 f_\lambda^\delta(a_t)_{1j}= \|v_1\| (\alpha+1)_r^{-1} Q_\delta(1+ i \lambda) (\sinh{t})^r (\cosh{t})^s \varphi_{\lambda}^{\alpha+r, \beta+s}(t) h_j^\delta(\lambda).
 \end{align}
 As the first order zeros of $Q_\delta(1+ i \lambda)$ are neutralized by that of $f_\lambda^\delta(a_t)_{1j}$, so we write:
 \begin{align}
 \label{ali:vvv}
 Q_\delta(1+ i \lambda)^{-1} f_\lambda^\delta(a_t)_{1j} = C_{v_1,\alpha} (\sinh{t})^r (\cosh{t})^s \varphi_{\lambda}^{\alpha+r, \beta+s}(t) h_j^\delta(\lambda).
 \end{align}
  The left hand side of (\ref{ali:vvv}) is holomorphic on $Int \mf a^*_\e$. To conclude that $h_j^\delta$ is  holomorphic
at $\lambda \in Int \mf a^*_\e$, we can choose $t_0>0$ so that $\varphi_{\lambda}^{\alpha+r, \beta+s}(t_0) \neq 0$ as is possible by the observation (\ref{Jacobi at zero-1}). Noting that $\varphi_\lambda^{\alpha+r, \beta+s}(t_0)$ is holomorphic in $\lambda$ and that both $\sinh{t_0}$ and $\cosh{t_0}$ are positive we reach our conclusion. To see that $\lambda \mapsto Q_\delta(1 - i \lambda)^{-1} h^\delta(\lambda)$ is holomorphic on $Int \mf a^*_\e$, we note that $f^\delta_\lambda(a_t)_{1j}$ is symmetric in $\lambda$ and so from (\ref{ali:vvv}) $Q_\delta(1 +i \lambda)^{-1}f^\delta_\lambda(a_t)_{1j} $  as well as $Q_\delta(1 -i \lambda)^{-1}f^\delta_\lambda(a_t)_{1j}$ are analytic in $Int \mf a^*_\e$. From the exact expression (\ref{Structure of the Kostant's polynomial}) of $ Q_\delta(1 - i \lambda)$  we further notice that the polynomials $Q_\delta(1 +i \lambda)$ and $Q_\delta(1 - i\lambda)$ have no common zeros. We can hence conclude that $[Q_\delta(1 - i\lambda)Q_\delta(1 + i\lambda)]^{-1} f^\delta_\lambda(a_t)_{1j}$ is analytic on $Int \mf a^*_\e$. Using (\ref{ali:vvv})
again we get the desired analyticity of $Q_\delta(1 - i \lambda)^{-1} h^\delta(\lambda)$
 in $Int \mf a^*_\e$.
\end{proof}
\begin{Remark}
\label{rem:continuous extension of h delta}
 For each $x \in X$ and $\delta \in \what{K}_M$; $\lambda \mapsto f_\lambda^\delta(x)$ extends as a continuous function to the closed strip $\mf a^*_\e$. From (\ref{structural form-2}) it follows that $\lambda \mapsto h^\delta(\lambda)$ also extends as a continuous function to $\mf a^*_\e$.
  \end{Remark}
 Our next aim is to  determine the decay of the function $h^\delta$, for that we need a lower bound of the associated Jacobi function.
\begin{Lemma}
 \label{Lemma of Bray}
(Bray \cite[Lemma 2.4]{Bray96})\\
Let $\mu , \tau \geq -\frac{1}{2} $, then for any $\Lambda > \frac{2}{\pi}$,
there is a constant $C$ depending on $\mu, \tau, \Lambda$ such that
\be
\label{Jakobi lower estimate}
\bigg|\varphi_\lambda^{\mu, \tau}\bigg(\frac{1}{|\lambda|^2}\bigg)\bigg|
\geq C_{\Lambda,\mu,\tau}, ~~\mbox{ for $|\lambda| > \Lambda$}
\ee
In fact the constant $C_{\Lambda,\mu,\tau}$ has the following form:
$$C_{\Lambda,\mu,\tau} = e^{-{{2 + \mu+ \tau} \over {\Lambda}}} ~\cos{\left( 1 / \Lambda \right)}.$$
\end{Lemma}
\begin{Remark}
\label{bray for SU(n,1) }
We note that for $G= SU(n, 1)$, the quantities  $\alpha \geq 0$,  $\beta=0$ (where $\alpha$, $\beta$ are as in (\ref{Structure of the Kostant's polynomial})) and the parameterization $(r, s)$ of $\what{K}_M$ runs over $\Z^+ \times \Z$ with $r \pm s \in 2\Z^+$. Suppose for some $\delta \in \what{K}_M$, $s_\delta < 0$. In such a case we use the relation \cite[(5.75)]{Koorn84}
\begin{equation}
\varphi_{\lambda}^{\alpha+ r_\delta, s_\delta}(t)= (\cosh{t})^{2|s_\delta|} \varphi_{\lambda}^{\alpha+r_\delta, |s_\delta|}(t), ~~ t>0, \lambda \in \C,
\end{equation}
and rewrite (\ref{Jacobi function}) as follows
\begin{equation}
\varphi_{\lambda, \delta}(t) = Q_\delta(i \lambda+1) (\alpha +1)^{-1}_{r_\delta}
(\sinh t)^{r_\delta} (\cosh t)^{|s_\delta|} \varphi_\lambda^{\alpha+r_\delta, |s_\delta|} (t).
\end{equation}
The Jacobi function $\varphi_\lambda^{\alpha+r_\delta, |s_\delta|} (t)$ clearly satisfies the conditions of Lemma \ref{Lemma of Bray}. For other classes of real rank one groups the parameters $\alpha+r$ and $\beta+s$ are positive integers. Thus, for all $G$ of real rank-one, the condition of Lemma \ref{Lemma of Bray} holds for the function $\varphi_{\lambda}^{\alpha+r, \beta+s}$.
\end{Remark}
\begin{Proposition}\label{Imp-pro}
For each $\delta \in \Gamma(f)$ the function $\lambda \mapsto h_i^\delta (\lambda)$
for each $i= 1, 2, \cdots, d_\delta$ satisfies the following decay condition:
\be\label{Conclusion}
\sup_{\lambda \in \mathfrak a^*_\e} \bigg| \bigg(\frac{d}
{d \lambda} \bigg)^n h_i^\delta(\lambda) \bigg|~(1+|\lambda|)^m < +\infty.
\ee
\end{Proposition}
\begin{proof}
The structural form  obtained in Lemma \ref{structure Lemma} and
(\ref{decay of operator valued projection})
gives the following decay/growth  condition for each $(1,j)$th matrix entry of $f_\lambda^\delta(a_t)$: for each $ m, n \in \Z^+ \cup \{0\}$
\begin{equation}
\label{decay of each matrix entry}
\sup_{a_t \in \mathfrak a^+; \lambda \in Int\mathfrak a^*_\e}
\bigg|\Phi_{\lambda, \delta~1}(a_t)~h_j^\delta(\lambda)  \bigg| (1+t)^n (1+|\lambda|)^m
\varphi_0^{-r_\e}(a_t)= c_{1j} < +\infty.
\end{equation}
This immediately implies that: for $\lambda \in Int \mf a^*_\e$ and $t> 0$
\begin{align}
\label{ali:20}
|\Phi_{\lambda, \delta~1}(a_t)| |h_j^\delta(\lambda)| (1 +|\lambda|)^m &\leq  c_{1j} \f{1}{(1+t)^n} \varphi_0^{r_\e}(a_t) \nonumber\\
&\leq c_{1j}(r_\e, t), \mbox{~~~~~where~} c_{1j}(r_\e, t) =\left\{
  \begin{array}{ll}
    c_{1j} & \hbox{if~~} r_\e \geq 0 \\
     c_{1j}e^{|r_\e|t}& \hbox{if~~} r_\e < 0
  \end{array}
  \right.
\end{align}
The last line of the inequality (\ref{ali:20}) is a consequence of the fact that $\varphi_0(a_t) < 1$ for all $t> 0$ and the two-sided estimate (\ref{eq:pri:apraori estimate of phi 0}) of $\varphi_0(a_t)$.\\
Now we express $\Phi_{\lambda, \delta~1}(a_t)$ in terms of the Jacobi function (\ref{Jacobi function}), which reduces (\ref{ali:20}) to the following:
\begin{align}
\label{ali:21}
|h^\delta_j(\lambda)| |\varphi_\lambda^{\alpha+r, \beta+s}(t)| |Q_\delta(i \lambda+1) (\sinh t)^r (\cosh{t})^s| ~(1+|\lambda|)^m \leq \f{1}{\|v_1\| }c_{1j}(r_\e, t).
\end{align}
We note that $h^\delta g^\delta(\lambda) Q_\delta(1 - i \lambda)$, $\lambda \in Int \mf a^*_\e$. Hence we get the inequality.
 \begin{align}
\label{ali:211}
|g^\delta_j(\lambda)| |\varphi_\lambda^{\alpha+r, \beta+s}(t)| |Q_\delta(1 - i \lambda) Q_\delta(i \lambda+1) (\sinh t)^r|& |(\cosh{t})^s| (1+|\lambda|)^m \leq \f{1}{\|v_1\| }c_{1j}(r_\e, t).
\end{align}
We now let $t = \f{1}{|\lambda|^2}$.
We choose one $\Lambda >\f{2}{\pi}$ large enough so that the disk $B^{\Lambda}(0)=\{ \lambda ~|~|\lambda| \leq \Lambda\}$ contains all the zeros of the polynomial $Q_\delta(1 + i \lambda)$ lying in $Int \mf a^*_\e$.
For $\lambda \in Int \mf a^*_\e \setminus B^{\Lambda}(0)$, by Lemma \ref{Lemma of Bray},
\begin{equation}
\lt|\varphi_\lambda^{\alpha+r, \beta+ s}\lt(\f{1}{|\lambda|^2}\rt)\rt| > C_{\Lambda,r,s}>0.
\end{equation}
 We note that the polynomial $Q_\delta(1+i \lambda) Q_\delta(1-i \lambda)$ is of degree $2 r$ (see (\ref{Structure of the Kostant's polynomial})). Hence for all  $\lambda \in Int \mf a^*_\e \setminus B^{\Lambda}(0)$ one can find a positive constant $\mf d$ such that
$ \lt|Q_\delta(1+i \lambda) Q_\delta(1-i \lambda) \lt(\sinh{\f{1}{|\lambda|^2}} \rt)^r \rt| > \mf d$.

 Also, for the above choice of $\lambda$, $e^{|r_\e|\f{1}{|\lambda|^2}} \leq e^{|r_\e|\f{1}{|\Lambda|^2}}$.
Hence from (\ref{ali:211}) we conclude that: for all $\lambda \in Int \mf a^*_\e \setminus B^{\Lambda}(0)$
\begin{equation}
|g^\delta_j(\lambda)| (1+|\lambda|)^m \leq \f{e^{|r_\e|\f{1}{|\Lambda|^2}}}{C_{\Lambda,\alpha+r,\beta+s} \mf d}=c_\delta (\mbox{say}) \mbox{~for each~} m \in \Z^+.
\end{equation}
As $ B^{\Lambda}(0)$ contains all the zeros of the polynomial $Q_\delta(1-i\lambda)$, so there exists a constant $k_\delta >0$ such that $|Q_\delta(1-i \lambda)|>k_\delta$ for all $\lambda \in Int \mf a^*_\e \setminus B^{\Lambda}(0)$. Thus, for $\lambda \in Int \mf a^*_\e \setminus B^{\Lambda}(0)$
\begin{align}
\label{ali:2111}
|h^\delta_j(\lambda)| (1+|\lambda|)^m = \f{1}{|Q_\delta(1-i\lambda)|}|g^\delta_j(\lambda)| (1+|\lambda|)^m
\leq \f{c_\delta}{k_\delta}.
\end{align}

For each $1 \leq j \leq d_\delta$ the following inequality is also obtained from (\ref{decay of operator valued projection}): for all $a_t \in A^+$ and $\lambda \in Int \mathfrak a^*_\e$
\begin{align}
&\hspace{-.5in}\lt| \lt(\frac{d}{d \lambda} \rt) \lt\{
\Phi_{\lambda, \delta~ 1}(a_t)Q_\delta(1-i\lambda)g_j^\delta(\lambda) \rt\} \rt| (1+t)^n (1+|\lambda|)^m
\varphi_0^{-r_\e}(a_t) \leq c'_{1j}. \nonumber\\
&\hspace{-.7in}\mbox{That is}\nonumber\\
& \hspace{-.7in}\lt| \lt\{\frac{d}{d \lambda}  g_j^\delta
(\lambda) \rt\} Q_\delta(1-i\lambda)\Phi_{\lambda, \delta~ 1}(a_t) + g_j^\delta(\lambda)~\lt(\frac{d}{d \lambda} \rt) \lt\{Q_\delta(1-i\lambda)\Phi_{\lambda, \delta~ 1}(a_t) \rt\}\rt|
  (1+|\lambda|)^m \leq c'_{1j}(r_\e, t).\nonumber\\
&\hspace{-.7in}\mbox{The last line can be written as}\nonumber\\
\label{Imp}
&\hspace{-.5in}\lt| \lt\{\frac{d}{d \lambda} g_j^\delta(\lambda) \rt\}
Q_\delta(1-i\lambda)\Phi_{\lambda, \delta~ 1}(a_t) \rt| (1+|\lambda|)^m \nonumber\\
& \hspace{1in}\leq c'_{1j}(r_\e, t)+ \lt| g_j^\delta(\lambda)\hspace{-.05in}\lt(\frac{d}{d \lambda} \rt)
\lt\{Q_\delta(1-i\lambda)\Phi_{\lambda, \delta~ 1}(a_t) \rt\} \rt| (1+|\lambda|)^m ,\nonumber\\
&\hspace{-.7in}\mbox{writing $\Phi_{\lambda, \delta~ 1}(a_t)$ in terms of the Jacobi
functions as  in (\ref{Jacobi function}) we get,}\nonumber\\
& \hspace{-.5in}\leq c'_{1j}(r_\e, t)+\f{(1+|\lambda|)^m}{(1+\alpha)_r} \lt| g_j^\delta(\lambda)\hspace{-.02in}
\lt(\hspace{-.03in}\frac{d}{d \lambda}\hspace{-.03in} \rt) \hspace{-.02in}\lt\{Q_\delta(1-i\lambda)Q_\delta(1+i \lambda)
(\sinh t)^r (\cosh t)^s \varphi_\lambda^{\alpha+r, \beta+ s} (t) \hspace{-.02in}\rt\} \rt|\nonumber\\
&\hspace{-.7in}\leq  c'_{1j}(r_\e, t)+ \f{(1+|\lambda|)^m}{(1+\alpha)_r}\lt| g_j^\delta(\lambda)
\lt\{ \hspace{-.04in}\lt(\hspace{-.03in}\frac{d}{d \lambda} \hspace{-.03in}\rt) Q_\delta(1-i\lambda)Q_\delta(1+i \lambda)\hspace{-.03in} \rt\}
(\sinh t)^r (\cosh t)^s \varphi_\lambda^{\alpha+r,\beta+ s} (t)  \rt|+
 \nonumber\\
&\hspace{.4in}\f{(1+|\lambda|)^m}{(1+\alpha)_r}\hspace{-.03in}\lt| g_j^\delta(\lambda)Q_\delta(1-i\lambda)Q_\delta(1+i \lambda)
(\sinh t)^r (\cosh t)^s \hspace{-.03in}\lt\{ \hspace{-.03in}\lt(\hspace{-.03in}\frac{d}{d \lambda} \hspace{-.03in}\rt) \hspace{-.02in}\varphi_\lambda^{\alpha+r,
\beta+s}(t)\hspace{-.03in}\rt\} \hspace{-.02in}\rt|.
\end{align}
We rewrite the inequality as
\begin{align}
\lt|\lt(\hspace{-.03in}\frac{d}{d \lambda}\hspace{-.03in} \rt) g_j^\delta(\lambda)  \rt|
(1+ |\lambda|)^m & <  c'_{1j}(r_\e, t) \lt[ |Q_\delta(1+i \lambda)Q_\delta(1-i\lambda)
\lt(\sinh{t}\rt)^r| \lt(\cosh{t}\rt)^s |\varphi_\lambda^{\alpha+ r, \beta+ s}(t)| \rt]^{-1}\nonumber \\
&\hspace{1.5in}+ |g_j^\delta(\lambda)| (1+ |\lambda|)^m
 \frac{|\f{d}{d \lambda}(Q_\delta(1-i\lambda)Q_\delta(1+ i \lambda)|}{|Q_\delta(1+i \lambda)Q_\delta(1-i\lambda)|}
 \nonumber\\
&\hspace{1.8in}+ |g_j^\delta(\lambda)| (1+ |\lambda|)^m \frac{\lt|\left(\frac{d}
{d \lambda} \right) \varphi_\lambda^{\alpha + r, \beta+s}\left(t\right)\rt|}{\lt| \varphi_\lambda^{\alpha + r, \beta+s}
\left( t\right)\rt|}.
\end{align}
Again we take $t = \f{1}{|\lambda|^2}$ and choose $\Lambda> \f{2}{\pi}$ suitably large so that $B^\Lambda(0)$ contains all the zeros of the polynomial $Q_\delta(1-i\lambda)$.
We will presently obtain bounds for the three terms on the right-hand side separately for $\lambda \in Int \mf a^*_\e \setminus B^\Lambda(0)$.
 The first term is bounded by a constant, as seen earlier. We have already obtained a bound for the factor $|g_j^\delta(\lambda)| (1+ |\lambda|)^m$ present in the last two terms. The quotient in the second term is clearly bounded in the region $Int \mf a^*_\e \setminus B^\Lambda(0)$. Finally, a bound for the quotient in the last term is obtained from the following facts   $|\varphi_\lambda^{\alpha+r, \beta+s}(\f{1}{|\lambda|^2})| \geq C_{\Lambda, \alpha+r, \beta+s}$ ( by Lemma \ref{Lemma of Bray}) and $\lt|\lt(\f{d}{d \lambda}\rt) \varphi_{\lambda}^{\alpha+r, \beta+s}(\f{1}{|\lambda|^2})  \rt| \leq c \f{1}{|\Lambda|^2} e^{(\e+\varrho)(\f{1}{|\Lambda|^2})} $ (by (\ref{Jacobi at zero-1})).

As $h^\delta_j(\lambda) = Q_\delta(1-i\lambda) g_j^\delta(\lambda)$, so for each $m \in \Z^+$
\be
\sup_{\lambda \in Int \mathfrak a^*_\e } \bigg|\bigg(\frac{d}{d
\lambda} \bigg) h_j^\delta(\lambda)  \bigg| ~(1+ |\lambda|)^m \leq
C_\delta < +\infty.
\ee
For any
order of the derivative on $\lambda$, we can use essentially the same argument.
\end{proof}
\begin{Corollary} \label{decay of h-delta}
From the decay (\ref{Conclusion}) of each matrix entry of the function
$h^\delta$ obtained in the above Proposition \ref{Imp-pro} we get: for
each fixed $\delta \in \Gamma(f)$, for each $m , n \in \Z^+ \cup \{0\}$
\be
\sup_{\lambda \in Int \mathfrak a^*_\e} \bigg\| \left( \frac{d}
{d \lambda}\right)^m h^\delta(\lambda) \bigg\|_{\mathbf{2}} ~(1+ |\lambda|)^n
< + \infty.
\ee
\end{Corollary}
Let us now recollect the properties we have obtained for the function $h^\delta$ ($\delta \in \what{K}_M$) in the following Lemma.
\begin{Lemma}
\label{lem:properties of h delta function}
The function $h^\delta$ (as obtained in (\ref{structural form-2})~)  is a $Hom(V_\delta, V_\delta^M)$ valued function on $\mf a^*_\e$ which satisfies the following properties:
\begin{enumerate}
  \item $h^\delta$ is holomorphic in interior of $\mf a^*_\e$ and it extends to the closed strip $\mf a^*_\e$ as a continuous function.
  \item $\lambda \mapsto Q_\delta(1- i \lambda)^{-1} h^\delta(\lambda)$ is an even function and also it is holomorphic on $Int \mf a^*_\e$.
  \item for each $m , n \in \Z^+ \cup \{0\}$
\be
\sup_{\lambda \in Int \mathfrak a^*_\e} \bigg\| \left( \frac{d}
{d \lambda}\right)^m h^\delta(\lambda) \bigg\|_{\mathbf{2}} ~(1+ |\lambda|)^n
< + \infty
\ee
\end{enumerate}
\end{Lemma}
The above Lemma shows that for each $\delta \in \what{K}_M$, $h^\delta \in \mc S_\delta(\mf a^*_\e)$ (see Definition \ref{def:ch2:delta SW space in the image side}~).
Hence by Theorem \ref{theo:ch2:main} the inversion $\mathcal I h^\delta$ given by
\begin{equation}
\label{Inversionofthedelta spherical transform}
\mathcal I h^\delta(x) =
\int_{\mathfrak a^{*+}} \Phi_{\lambda, \delta}(x) h^\delta(\lambda)
|\textbf{c}(\lambda)|^{-2} d \lambda,
\end{equation}
belongs to  the $p$th Schwartz class ${\mc S^{p}}^\delta(X)$ where $p$ is determined by the chosen $\e$ by $p = \frac{2}{1+ \e}$.\\
It is clear from (\ref{Inversionofthedelta spherical transform}), (\ref{structural form-2}) and (\ref{Inversion}) that $f^\delta \equiv \sqrt{d_\delta}~ \mc I h^\delta $. Hence for each $\delta$, the function $trf^\delta$ satisfies the Schwartz space decay condition:  for each $\textbf{D},\textbf{ E} \in \mathcal
U(\mathfrak g)$ and $n \in \Z^+ \cup\{0\}$,
\begin{equation}
\label{OO} \sup_{x \in X} |tr f^\delta(\textbf{D}; x;\textbf{ E}) |(1+|x|)^n
\varphi_0^{-\frac{2}{p}}(x) \leq + \infty.
\end{equation}
The function $f$
obtained in (\ref{Inversion}) was proved to be left-$K$-finite. Hence $f \in \mc S^p(X)$ where  $p= \f{2}{1+\e}$.

Finally we shall show that $P_\lambda f(x) = f_\lambda(x)$
(~$\lambda \in \mathfrak a^*_\e$~) where the function $f$ is obtained in (\ref{Inversion}). As $f$ is left $K$ finite $P_\lambda f$
can be decomposed as follows
\begin{align}\label{final decomposition}
P_\lambda f(x) &=\sum_{\delta \in \Gamma(f)}tr(P_\lambda f)^\delta(x) \nonumber\\
&=\sum_{\delta \in \Gamma(f)}trP_\lambda ( f)^\delta(x),
\end{align}
where the last line follows by using Proposition \ref{delta projection of
P-lambda}, and $\Gamma(f)$ being a finite subset of $\widehat{K}_M$. Now from the above discussion and  Proposition \ref{delta projection of P-lambda} we get the following:
\begin{equation}
P_\lambda ( f)^\delta(x) = \Phi_{\lambda, \delta} (x)
\widetilde{f^\delta}(\lambda)=  \Phi_{\lambda, \delta} (x)
h^\delta(\lambda)= f_\lambda^\delta(x),~~~\lambda \in \mathfrak a^*_\e.
\end{equation}
Thus (\ref{final decomposition}) can be reformulated as $P_\lambda f(x) = \sum_{\delta \in \Gamma(f)}tr
f_\lambda^\delta(x) $,
which immediately gives $P_\lambda f(x) = f_\lambda(x)$ for all
$\lambda \in \mathfrak a^*_\e$ and $x \in X$. We give the gist
of what we have shown in this section in the form of the following
theorem
\begin{Theorem}
Any continuous function $g: (x, \lambda) \mapsto g_\lambda(x)$
defined on $X \times \mathfrak a^*_\e$ with certain $\e \geq 0 $ and satisfying the conditions of Definition \ref{def:K-finite subspace of teh image under sp-projc} is of the form $g_\lambda(x) =P_\lambda f(x) $ for some
left-$K$-finite function $f \in \mc S^p(X)$ where $p = \frac{2}{1 + \e}$ .
\end{Theorem}

\section{\textbf{Inverse Paley-Wiener Theorem}}
\label{sec: Inverse Paley-Wiener Theorem}

We begin the section with a definition.
\begin{Definition}
\label{def:P-R-X}
Let $P_R(X)$ be the class of scalar valued functions $(\lambda, x) \mapsto f_\lambda(x)$ on $\mf a^* \times X$ satisfying:
\begin{enumerate}
  \item the map $\lambda \mapsto f_\lambda(\cdot)$ is an even,  compactly supported $\mc C^\infty$ function on $\mf a^*$ with it support lying in $[-R, R]$ (note that our group $G$ is of real rank-one and thus we have identified $\mf a^*$ with $\R$),
  \item  for each $\lambda \in \mf a^*$, $x \mapsto f_\lambda(x)$ is a $C^\infty$ function on $X$ and $f_\lambda(\cdot) \in \mc E_\lambda(X)$.
\end{enumerate}

\end{Definition}
In this section we shall try to characterize the space $P_R(X)$ as an image of certain subspace of $L^2(X)$ under the spectral projection. \\
We need to recall some basic results regarding a certain $G$-invariant domain $\Xi$ in $G_\C/K_\C$ called the \emph{complex crown}.  Here $X_\C = G_\C/K_\C$ is the natural  complexification (see \cite{krotz-Olafsson-Stanton05}) of the symmetric space $X$. The domain can be explicitly written as $\Xi = G \exp{i \Omega}\cdot x_0 $, where $x_0 = eK$ and $\Omega = \{ H \in \mf a~|~ |\alpha(H)| < \frac{\pi}{2}, \alpha \in \Sigma \}$ \cite{Krotz-Stanton-05}. Let $\mc G(\Xi) $ be space of all holomorphic functions on the complex crown. For $\lambda \in i \mf a^*$, the function $H \mapsto \varphi_{\lambda}(\exp{iH})$ ($H \in \mf a$) can be analytically continued in the tube domain $\mf a + 2 i \Omega$ \cite{Krotz-Stanton-04}.  Almost all the basic analysis on the crown domain uses a fundamental tool called the \emph{orbital integrals} developed by Gindikin et al. \cite{Krotz-Gindikin02}. Let $h$ be a function on $\Xi$ suitably decreasing at the boundary and $Y \in 2 \Omega$, then the orbital integral is defined by
\begin{equation}
\label{eq: orbital integral}
O_{h}(iY) = \int_G h\left( g \exp{\left(\frac{i}{2}Y \right)}\cdot x_0 \right)~dg.
\end{equation}
If a holomorphic function $\theta$ on the tube $\mf a + 2 i \Omega$ has the representation
$$\theta(Y) = \int_{i \mf a^*} g(\lambda)~\varphi_\lambda(\exp{Y}) |\textbf{c}(\lambda)|^{-2} d\lambda,$$ then we define (with certain condition on $g$, see \cite{krotz-Olafsson-Stanton05}) the operator
$$D\theta(Y)=\int_{i \mf a^*}g(\lambda) \psi_\lambda(Y)  |\textbf{c}(\lambda)|^{-2} d\lambda,$$ where $\psi_\lambda(Y) =e^{\langle\lambda,Y \rangle}+e^{\langle\lambda,-Y \rangle}$. The operator $D$ is a pseudo-differential shift operator \cite{krotz-Olafsson-Stanton05}. The following is an inverse Paley-Wiener theorem for the Helgason Fourier transform.
\begin{Theorem}
\label{theo-thangavelu}
\emph{[Thangavelu \cite[Theorem 2.3]{Veluma07}]}\\
Let $f \in L^2(X)$, then the Helgason Fourier transform $\mc F f(\lambda, kM)$ is supported in $|\lambda| \leq R$ if and only if the function  $f$ has a holomorphic extension $F \in \mc G(\Xi)$ which satisfies the estimate
\begin{equation}
\label{eq:holo-estimate}
DO_{|F|^2} (iY) \leq C e^{2 R|Y|} ,~~\mbox{~where $C$ is independent of $Y \in 2 \Omega$.}
\end{equation}
\end{Theorem}
Our main theorem in this section is a consequence of the above theorem.
\begin{Theorem}
A function $f_\lambda(x)$ ($ x \in X,~\lambda \in \mf a^*$) is in  $P_R(X)$ if and only if
 $f_\lambda(x) = (f \ast \varphi_\lambda)(x)$ ($\forall~ x \in X,~\lambda \in \mf a^*$) for some $f \in L^2(X)$ which admits a holomorphic extension $F \in \mc G(\Xi)$ satisfying the estimate (\ref{eq:holo-estimate})
\end{Theorem}
\begin{proof}
Let $f_\lambda(x) \in P_R(X)$; then we get a function
\begin{equation}
\label{eq-sec-4-1}
f(x)  = \int_{\mf a^{*+}} f_\lambda(x) ~|\textbf{c}(\lambda)|^{-2} d\lambda.
\end{equation}
The integral (\ref{eq-sec-4-1}) is obviously convergent and $f \in \mc C^\infty(X)$. A simple application of the Peter-Weyl theorem gives
\begin{equation}
\label{eq-sec-4-2}
f(x)=\int_{\mf a^{*+}} \left[ \sum_{\delta \in \what{K}_M} tr~ f_\lambda^\delta(x)\right]~|\textbf{c}(\lambda)|^{-2} d\lambda.
\end{equation}
Now for each $\delta \in \what{K}_M$ and $\lambda \in  \mf a^*$, $f_\lambda^\delta \in \mc E_\lambda^\delta(X)$, hence by Theorem \ref{structural form-2} we can write $f_\lambda^\delta(x) = \sqrt{d_\delta} \Phi_{\lambda, \delta}(x) h^\delta(\lambda)$. As $\lambda \mapsto f_\lambda^\delta$ is compactly supported and $\Phi_{\lambda, \delta}$ is an entire function so the function $h^\delta$ must have its support in $[-R, R]$. The above structural form can further reduce (\ref{eq-sec-4-2}) as follows:
\begin{align}
\label{sec-4-ali-1}
f(x) & = \int_{\mf a^{*+}} \left[ \sum_{\delta \in \what{K}_M } \sqrt{d_\delta} \sum_{i=1}^{d_\delta} \langle \Phi_{\lambda, \delta}(x)h^\delta(\lambda) v_i, v_i \rangle \right] ~|\textbf{c}(\lambda)|^{-2} d\lambda \nonumber\\
&= \int_{\mf a^{*+}} \left[ \sum_{\delta \in \what{K}_M } \sqrt{d_\delta}\sum_{i=1}^{d_\delta} \langle \Phi_{\lambda, \delta}(x)v_1, v_i \rangle\langle h^\delta(\lambda)v_i,v_1 \rangle \right] |\textbf{c}(\lambda)|^{-2} d\lambda \nonumber\\
&=\int_{\mf a^{*+}} \left[ \sum_{\delta \in \what{K}_M } \sqrt{d_\delta} \sum_{i=1}^{d_\delta} \langle \int_K e^{-(i \lambda +1)H(x^{-1}k) } \delta(k) dk~ v_1, v_i \rangle~   \langle h^\delta(\lambda) v_i,v_1 \rangle \right] |\textbf{c}(\lambda)|^{-2} d\lambda\nonumber\\
&=\int_{\mf a^{*+}}\int_K  e^{-(i \lambda +1)H(x^{-1}k) }\left[ \sum_{\delta \in \what{K}_M } \sqrt{d_\delta} \sum_{i=1}^{d_\delta} \langle  \delta(k)~v_1, v_i \rangle~\langle h^\delta(\lambda) v_i,v_1 \rangle   \right]|\textbf{c}(\lambda)|^{-2} dk~ d\lambda\nonumber\\
&= \int_{\mf a^{*+}}\int_K  e^{-(i \lambda +1)H(x^{-1}k) } \left[ \sum_{\delta \in \what{K}_M} \sqrt{d_\delta} ~tr(\delta(k ) h^\delta(\lambda)) \right]|\textbf{c}(\lambda)|^{-2} dk d\lambda.
\end{align}
We denote the function $h(\lambda, k) = \sum_{\delta \in \what{K}_M} \sqrt{d_\delta}~ tr(\delta(k ) h^\delta(\lambda))$ in (\ref{sec-4-ali-1}). Then clearly it is a $\mc C^\infty$ function in the $\lambda$ variable and  the function
$\lambda \mapsto \int_K  e^{(-i \lambda +1)H(x^{-1}k) }h(\lambda, k) dk  $
is  even. Also $h(\lambda, \cdot)$ is a compactly supported function and $f$ is nothing but the Helgason Fourier inversion of the function $h$. Hence by \cite[Theorem 2.3]{Veluma07}, the function $f \in L^2(X)$ and it admits a holomorphic extension on the complex crown  satisfying the estimate (\ref{eq:holo-estimate}).\\
On the other hand if $g \in L^2(X)$ then, for all $\lambda\in\mf a^*$ and $x \in X$,  $$P_\lambda g(x)= g \ast \varphi_\lambda(x)=\int_{K} e^{-(i \lambda + 1) H(x^{-1}k)} \widetilde{g}(\lambda, k)~dk.$$ Furthermore if $g$ can be extended holomorphically to some $\overline{g} \in \mc G(\Xi)$ with $\overline{g}$ satisfying (\ref{eq:holo-estimate}) then by Theorem \ref{theo-thangavelu},  $\widetilde{g}(\lambda, \cdot)$ is supported in $[-R, R]$. It is now easy to show that $g \ast \varphi_\lambda(\cdot) \in P_R(X)$.
\end{proof}
\providecommand{\bysame}{\leavevmode\hbox to3em{\hrulefill}\thinspace}
\providecommand{\MR}{\relax\ifhmode\unskip\space\fi MR }
\providecommand{\MRhref}[2]{%
  \href{http://www.ams.org/mathscinet-getitem?mr=#1}{#2}
}
\providecommand{\href}[2]{#2}

\end{document}